\documentclass[12pt, final]{l4dc2022}

\title[Piecewise Learning and Control with Stability Guarantees]{A Piecewise Learning Framework for Control of Unknown Nonlinear Systems with Stability Guarantees}
\usepackage{times}
\usepackage{graphicx}

\usepackage{mathrsfs}
\usepackage{amsmath,amssymb,amstext} 

\allowdisplaybreaks

\usepackage{pgfplots}
\pgfplotsset{compat=1.8}

\usepgfplotslibrary{fillbetween}  
\usepackage{xcolor}

\newtheorem{assumption}{\textbf{Assumption}}

\author{%
 \Name{Milad Farsi} \Email{mfarsi@uwaterloo.ca}\\
 \addr Department of Applied Mathematics, University of Waterloo, Canada 
 \AND
 \Name{Yinan Li}\Email{yinan.li@uwaterloo.ca} \\
 \addr Clearpath Robotics, Kitchener, Canada
 \AND
 \Name{Ye Yuan}
 \Email{yye@hust.edu.cn} \\
 \addr School of Artificial Intelligence \& Automation, Huazhong University of Science and Technology, China
 \AND
 \Name{Jun Liu} \Email{j.liu@uwaterloo.ca}\\
 \addr Department of Applied Mathematics, University of Waterloo, Canada 
}

\begin{document}

\maketitle

\begin{abstract}%
 We propose a piecewise learning framework for controlling nonlinear systems with unknown dynamics. While model-based reinforcement learning techniques in terms of some basis functions are well known in the literature, when it comes to more complex dynamics, only a local approximation of the model can be obtained using a limited number of bases. The complexity of the identifier and the controller can be considerably high if obtaining an approximation over a larger domain is desired. To overcome this limitation, we propose a general piecewise nonlinear framework where each piece is responsible for locally learning and controlling over some region of the domain. We obtain rigorous uncertainty bounds for the learned piecewise models. The piecewise affine (PWA) model is then studied as a special case, for which we propose an optimization-based verification technique for stability analysis of the closed-loop system. Accordingly, given a time-discretization of the learned {PWA} system, we iteratively search for a common piecewise Lyapunov function in a set of positive definite functions, where a non-monotonic convergence is allowed. This Lyapunov candidate is verified on the uncertain system to either provide a certificate for stability or find a counter-example when it fails. This  counter-example is added to a set of samples to facilitate the further learning of a Lyapunov function. We demonstrate the results on two examples and show that the proposed approach yields a less conservative region of attraction (ROA) compared with alternative state-of-the-art approaches. Moreover, we provide the runtime results to demonstrate potentials of the proposed framework in real-world implementations. 
\end{abstract}

\begin{keywords}%
  Nonlinear systems, learning-based control, system identification, piecewise affine, Lyapunov analysis
\end{keywords}

\section{Introduction}
The flexibility of piecewise affine (PWA) systems makes them suitable for different approaches in control. Hence, the control problem of piecewise systems has been extensively studied in the literature (see, e.g.,  \cite{marcucci2019mixed,zou2007robust,rodrigues2005piecewise,baotic2005optimal,strijbosch2020monotonically,christophersen2005optimal,rodrigues2003observer}). Moreover, various applications can be found for PWA systems, including robotics \citep{andrikopoulos2013piecewise,marcucci2017approximate}, automotive control  \citep{borrelli2006mpc,sun2019hybrid}, and power electronics \citep{geyer2008hybrid,vlad2012explicit}. 
Since PWA systems are highly adaptable to different problems, various techniques are presented to efficiently fit a piecewise model to data (see, e.g.,  \cite{toriello2012fitting,breschi2016piecewise,ferrari2003clustering,amaldi2016discrete,rebennack2020piecewise,du2021online}). A review of some of the techniques can be found in \cite{gambella2021optimization,garulli2012survey}. 

Solving the optimal control problem for PWA systems has also been the topic of various works. In \cite{zou2007robust}, the robust model predictive control (MPC) strategy is extended to PWA systems with polytopic uncertainty, where multiple PWA quadratic Lyapunov functions are employed for different vertices of the uncertainty polytope in different partitions. In another work by \cite{marcucci2019mixed}, hybrid MPC is formulated as a mixed-integer program to solve the optimal control problem for PWA systems. However, these techniques are only available in the open-loop form, which decreases their applicability for real-time control. 

Deep neural networks (DNN) offer an efficient technique for control in the closed loop. However, one drawback of DNN-based control is the difficulty in stability analysis. This becomes even more challenging when PWA are considered. \cite{chen2020learning} suggested a sample-efficient technique for synthesizing a Lyapunov function for the PWA system controlled through a DNN in the closed loop. In this approach, the analytic center cutting plane method (ACCPM)  \citep{goffin1993computation,nesterov1995cutting,boyd2004convex} is first used for searching for a Lyapunov function. Then, this Lyapunov candidate is verified on the closed-loop system using mixed-integer quadratic programming (MIQP). This approach relies on our knowledge of the exact model of the system and therefore cannot be directly implemented on an identified PWA system with uncertainty.   

Although learning with guarantees has motivated plenty of recent research  \cite{chang2020neural,chen2020learning,dai2021lyapunov}, fewer works have explicitly considered the model uncertainty. In \cite{berkenkamp2016safe}, the authors consider an approach that learns the region of attraction (ROA)  from experiments on a partially unknown system. Based on regularity assumptions on the model errors in terms of a Gaussian process (GP) prior, they employ an underlying Lyapunov function to determine an ROA from which the system is asymptotically stable with high probability. Even though a partially unknown model including uncertainty is considered in this approach, the approach may not computationally scale well with the number of samples considering the use of a GP learner. 

Model-based learning approaches in the literature are mainly categorized under reinforcement learning (RL) in two groups: value function and policy search methods. Approximate/adaptive dynamic programming techniques \citep{wang2009adaptive,lewis2009reinforcement,balakrishnan2008issues} as a well-known value-based approach can efficiently approximate the solution to the optimal control problem. Even though a set of polynomial bases, for instance, is known to be sufficient as a universal approximator over a compact domain  \citep{SOL:kamalapurkar2018model}, the number of bases required for a tight approximation of the dynamics over a given domain may be exceedingly high. This highly impedes implementations, especially in an online learning and control setting.

On the other hand, employing a piecewise approach can improve the applicability of model-based learning  greatly by keeping the online computations needed for updating the model and control in a tractable size, since at any instance only a particular mode of the system is involved. Hence, in this paper, we consider a partition of the domain consisting of different pieces, where for each piece, we run a local learner in terms of a limited number of bases using a structured online learning (SOL) approach \citep{farsi2020structured,farsi2021structured,farsi2022Quad} to obtain a piecewise feedback control. Then, by considering the special case of PWA systems, an optimization-based technique is employed to verify a variant of the piecewise model to obtain rigorous stability guarantees. Two examples are shown to demonstrate the advantages of the proposed framework in terms of providing less conservative stability guarantees in terms of the size of the verified ROA, compared with Lyapunov functions learned using neural networks for systems with known dynamics \citep{chang2020neural}.

\section{Problem Formulation} \label{sec2}
	Consider the nonlinear system in control-affine form
\begin{align} \label{SOL_sys}
\dot{x}=F(x,u)=f(x)+g(x)u = f(x)+\sum_{j=1}^m g_j(x)u_j, 
\end{align}
where $x\in D\subset \mathbb{R}^n$, $u\in\Omega\subset\mathbb{R}^m$, $f:\,D\rightarrow \mathbb{R}^n$, and $g:\,D\rightarrow \mathbb{R}^{n\times m}$. 

The cost functional to be minimized along the trajectory, starting from the initial condition $x(0)=x_0$, is considered to be in the following linear quadratic form
\begin{align} \label{SOL_cost_x}
J(x_0,u)=\lim_{T\rightarrow\infty}\int_{0}^T \mathrm{e}^{-\gamma t}\left(x^TQx+u^TRu\right) \mathrm{dt},
\end{align}
where $Q\in\mathbb{R}^{n\times n}$ is positive semi-definite, $\gamma\geq 0$ is the discount factor, and $R\in\mathbb{R}^{m\times m}$ is a diagonal matrix with only positive values, given by the design criteria.

\section{The Piecewise Learning and Control Framework} \label{sec3}

We approximate the nonlinear system (\ref{SOL_sys}) 
by a piecewise model with a bounded uncertainty
\begin{align} \label{PW_PW_sys_approx}
\dot{x}=W_\sigma\Phi(x)+\sum_{j=1}^{m} W_{j\sigma}\Phi(x)u_j+d_\sigma,
\end{align}
where $d_\sigma\in \mathbb{R}^n$ is a time-varying uncertainty,  $W_\sigma$ and $W_{j\sigma }\in\mathbb{R}^{n\times p}$ are the matrices of the coefficients for $\sigma \in \{1,2,\dots,n_\sigma\}$ and $j \in \{1,2,\dots,m\}$, with a set of differentiable bases $\Phi(x)=[ \phi_{1}(x) \quad\dots\quad \phi_{p}(x)]^T$, and $n_\sigma$ denoting the total number of pieces. 
Moreover, any piece of the system is defined over a convex set given by a set of linear inequalities as
$
\Upsilon_\sigma=\{x \in D| Z_\sigma x \leq z_\sigma \},
$
where $\sigma \in \{1,\dots,n_\sigma\}$ and $Z_\sigma$ and $z_\sigma$ are a matrix and a vector, respectively, of appropriate dimensions.

We assume that the set  $\{\Upsilon_\sigma\}$ forms a partition of the domain and its elements do not share any interior points, i.e. $\bigcup_{\sigma=1}^{n_\sigma}\Upsilon_\sigma= D$ and $\text{int}[\Upsilon_\sigma] \bigcap \text{int}[\Upsilon_l]= \emptyset$ for $\sigma \neq l$ and $\sigma$, $l \in \{1,2,\dots,n_\sigma\}$. Furthermore, the piecewise model is assumed to be continuous across the boundaries of $\{\Upsilon_\sigma\}$ (see Appendix \ref{app:continuity}). The control input and the uncertainty are assumed to be bounded and lie in the  sets
$
\Omega=\{u \in\mathbb{R}^m \arrowvert |u_j| \leq \bar{u}_j,\,\forall j \in \{1,2,\dots,m\}\}
$
and
$\Delta_\sigma=\{d_\sigma \in\mathbb{R}^n \arrowvert |d_{\sigma i}| \leq \bar{d}_{\sigma i},\,\forall i \in \{1,2,\dots,n\}\},
$ respectively. The uncertainty upper bound $\bar{d}_{\sigma }=(\bar{d}_{\sigma 1},\cdots,\bar{d}_{\sigma n})$ is to be determined. 

\subsection{System Identification} \label{PW_section:PW_identify}

Having defined the parameterized model of the system, we employ a system identification approach to update the system parameters. For each pair of samples obtained from the input and state of the system, i.e., $(x^s,u^s)$, we first locate the element in the partition $\{\Upsilon_\sigma\}$ that contains the sampled state $x^s$. Then, we locally update the system coefficients of the particular piece from which the state is sampled. In \cite{farsi2020structured}, the weights are updated according to
\begin{align}\label{PW_learning_cost}
{[\hat W_\sigma\quad \hat W_{1\sigma}\;\dots\; \hat W_{m\sigma}]}_{k}= \underset{\bar W}{\arg\min} \quad \|\dot X_{k\sigma}-\bar W\Theta_{k\sigma}\|^2_2,
\end{align} 
where $k$ is the time step, and $\Theta_{k\sigma}$ includes a matrix of samples with  
\begin{align*}
{\Theta_k}^s={[\Phi^T(x^s)\quad \Phi^T(x^s)u_1^s\quad \dots \quad \Phi^T(x^s)u_m^s]}^T_k,
\end{align*}
for the $s$th sample in the $\sigma$th partition. Correspondingly, $\dot X_{k\sigma}$ contains the sampled state derivatives. While in principle any identification technique can be used, e.g., \cite{brunton2016discovering,yuan2019data}, 
the linearity with respect to the coefficients allows us to employ least-squares techniques. In this paper, we implement the recursive least-squares (RLS) \citep{ljung1983theory,liu2016efficient,wu2015recursive} technique that provides a more computationally efficient way to update the parameters. Accordingly, only one sample at each time is used to update the weights, instead of processing a history of samples. 

\subsection{Feedback Control}\label{PW_section:PW_control}

 In \cite{farsi2020structured}, a matrix differential equation is proposed using a quadratic parametrization in terms of the basis functions to obtain a feedback control. Here, we adopt a similar learning framework, but consider a family of $n_\sigma$ differential equations, each of which corresponds to one particular mode of the system in the piecewise model. We integrate the following state-dependent Riccati differential equation in forward time: 
 \begin{align} \label{PW_PW_P_ode}
-\dot{P_\sigma}=&\bar{Q}+P_\sigma\frac{\partial{\Phi(x)}}{\partial{x}}W_\sigma+W_\sigma^T\frac{\partial{\Phi(x)}}{\partial{x}}^TP_\sigma-\gamma P_\sigma\nonumber\\ 
&-P_\sigma\frac{\partial{\Phi(x)}}{\partial{x}}\bigg(\sum_{j=1}^{m}W_{j\sigma}\Phi(x)r_j^{-1}\Phi(x)^TW_{j\sigma}^T\bigg)\frac{\partial{\Phi(x)}}{\partial{x}}^TP_\sigma.
\end{align}
The solution to the differential equation (\ref{PW_PW_P_ode}) characterizes the value function defined by
 \begin{align} \label{value} 
 V_\sigma=\Phi^TP_\sigma \Phi,
 \end{align}
based on which we obtain a piecewise control
\begin{align} \label{PW_PWA_control}
u_j=-r_j^{-1}\frac{\partial V_\sigma}{\partial{x}}^T g_j(x)=-\Phi(x)^Tr_j^{-1}P_\sigma\frac{\partial{\Phi(x)}}{\partial{x}}W_{j\sigma}\Phi(x).
\end{align}

\section{Analysis of Uncertainty Bounds}\label{sec4}

We use the uncertainty in the piecewise system (\ref{PW_PW_sys_approx}) to capture approximation errors in identification. In this section, we analyze the worst-case bounds to provide  guarantees for the proposed framework. 

There exist two sources of uncertainty that affect the accuracy of the identified model. The first is the mismatch between the identified model and the observations made. The latter may also be affected by the measurement noise. The second is due to unsampled areas in the domain. We can estimate the uncertainty bound for any piece of the model by combining these two bounds. In what follows, we discuss the procedure of obtaining these bounds in more detail.

\begin{assumption} \label{PW_assu_Lip_e}
	For any given $(x^s,u^s)$, let $F_i(x^s,u^s)$ be the $i$th element of $F(x^s,u^s)$. We assume that ${F}_i(x^s,u^s)$ can be measured with some tolerance as $\tilde{F}_i(x^s,u^s)$, where $|\tilde{F}_i(x^s,u^s)-{F}_i(x^s,u^s)|\leq \varrho_e|\tilde{F}_i(x^s,u^s)|$ with $0\leq\varrho_e< 1$ for all $i\in\{1,\cdots,n\}$.
\end{assumption}

We make predictions $\hat{F}_i(x^s,u^s)$ of the state derivatives for any sample using the identified model. Hence, we can  easily compute the distance between the prediction and the approximate evaluation of the system by using the samples collected for any piece. This gives the loss $|\hat{F}_i(x^s,u^s)-\tilde{F}_i(x^s,u^s)|$. The proof of the following result can be found in Appendix \ref{app:thm1}.

\begin{theorem}\label{thm:bound1}
	Let Assumption \ref{PW_assu_Lip_e} hold, and $S_{\Upsilon\sigma}$ denote the set of indices for sample pairs  $(x^s,u^s)$ such that $x^s \in \Upsilon_\sigma$. Then, an upper bound of the prediction error, regarding any sample $(x^s,u^s)$ for $s \in \{1,\dots,N_s\}$, is given by
	\begin{align*}
	|\hat{F}_i(x^s,u^s)-{F}_i(x^s,u^s)|&\leq  \bar d_{e\sigma i}:=\underset{s\in S_{\Upsilon\sigma}}{\max}(|\hat{F}_i(x^s,u^s)-\tilde{F}_i(x^s,u^s)|+\varrho_e |\tilde{F}_i(x^s,u^s)| ),
	\end{align*}
	where $\sigma \in \{1,\dots,n_\sigma\}$, and $i \in \{1,\dots,n\}$.
\end{theorem}

The samples may not be uniformly obtained from the  domain. Depending on how smooth the dynamics are, there might be unpredictable behavior of the system in the gaps among the samples. Hence, the predictions made by the identified model may be misleading in the areas we have not visited yet. To take this into account, we assume a Lipschitz constant is given for the system. More specifically, we let $\varrho_x\in\mathbb{R}^{n}_+$ and $\varrho_u\in\mathbb{R}^{n}_+$ denote the Lipschitz constants of $F(x,u)$ with respect to $x$ and $u$ on $D\times \Omega$, respectively. 
We use this to bound the uncertainty for the unsampled areas.

The procedure starts with searching for the largest gaps in the state and control spaces that do not contain any samples as described in detail in Appendix \ref{app:qperror}. Let  $(x^{s*},u^{s*})$ be the closest sample indexed in $S_{\Upsilon\sigma}$ to the center point $(c^*_{x\sigma},c^*_{u\sigma})$ of the sample gap (as a Euclidean ball) with radii  $(r_{x\sigma}^*,r_{u\sigma}^*)$. We need to compute the worst case of the prediction error at the center point that is given by $|\hat{F}_i(c^*_{x\sigma},c^*_{u\sigma})-{F}_i(c^*_{x\sigma},c^*_{u\sigma})|$, where $\hat{F}(\cdot,\cdot)$ denotes an evaluation of the identified model. However, according to Assumption \ref{PW_assu_Lip_e}, we do not have access to the original system to exactly evaluate ${F}(\cdot,\cdot)$. Therefore, we obtain the bound in terms of the approximate value instead.
\begin{theorem}\label{thm:bound2}
	Let Assumptions \ref{PW_assu_Lip_e}-\ref{PW_assu_Lip_estim} hold and $(r^*_{x\sigma},r^*_{u\sigma})$ be given by the solutions of (\ref{PW_gap_center_x}) and (\ref{PW_gap_center_u}) (details given in Appendix \ref{app:qperror}). Then, an upper bound for the prediction error can be obtained regarding all unvisited points $x \in \Upsilon_\sigma$ and $u \in \Omega$ as below
	\begin{align}
	|F_i(x,u)-\hat{F}_i(x,u)|\leq \bar d_{\sigma i}=\varrho_{ui} r_{u\sigma}^* + \varrho_{xi} r_{x\sigma}^*+ \bar d_{e\sigma i}+\hat \varrho_{ui} r_{u\sigma}^*+\hat \varrho_{xi} r_{x\sigma}^* \label{PW_bound_gap}.
	\end{align}
\end{theorem}

The proof can be found in Appendix \ref{app:thm2}. 

\section{Stability Verification for Piecewise-Affine Learning and Control} \label{sec5}

\subsection{Piecewise Affine Models}

A special case of system  (\ref{PW_PW_sys_approx}) can be obtained when we choose $\Phi(x)=\begin{bmatrix}
1 &x^T
\end{bmatrix}$. 

We consider system coefficients in the form of $W_\sigma=\begin{bmatrix}
C_\sigma & A_\sigma
\end{bmatrix}$ and $W_{j\sigma}=\begin{bmatrix}
B_{j\sigma} & 0
\end{bmatrix}$. Clearly, $A_\sigma$, $B_{j\sigma}$, and $C_\sigma$ can be used to rewrite the PWA system in the standard form 
\begin{align}\label{PW_PW_lin_sys_approx1}
\dot{x}=A_\sigma x+\sum_{j=1}^{m} B_{j\sigma} u_j+C_\sigma+d_\sigma,
\end{align}

\subsection{MIQP-based Stability Verification of PWA Systems}

In this section, we adopt an MIQP-based verification technique based on the approach presented in \cite{chen2020learning}. In this framework, by considering a few steps ahead, we verify that the Lyapunov function is decreasing. However, it may not be necessarily monotonic, meaning that it may be increasing in some steps and then be decreasing greatly in some other steps to compensate. Regarding the fact that this approach is inherently a discrete technique, we need to consider a discretization of (\ref{PW_PW_lin_sys_approx1}). By an Euler approximation, we have
\begin{align}  \label{PW_PW_lin_sys_approx_disc}
x_{k+1}=\check{F}_{d} (x_k,u_k)=\check{A}_\sigma x_k+\sum_{j=1}^{m} \check{B}_{j\sigma}u_{jk}+\check{C}_\sigma+d_{\sigma}, 
\end{align}
where $\check{A}_\sigma$, $\check{B}_{j\sigma}$, and $\check{C}_\sigma$ are the discrete system matrices of the same dimension as (\ref{PW_PW_lin_sys_approx1}). Moreover, we re-adjust the uncertainty bound as $\bar{d}_\sigma:=h\bar{d}_\sigma$, where $h$ denotes the time step. 
 
We refer the uncertain closed loop system with the control $u_{jk}=\omega_j(x_k)$ as
\begin{align} \label{PW_PW_dis_lin_cl}
x_{k+1}=\check{F}_{d,cl} (x_k).
\end{align}
For this system, let the convex set $\bar D=\{x\in D | Z_{\bar D} x \leq z_{\bar D}\}$ be a user-defined region of interest (ROI), within which obtaining a region of attraction (ROA) is desirable. 

\subsubsection{Learning and Verification of A Lyapunov Function}

Assuming $u_j=-r_j^{-1}B_{j\sigma}^TP_{3\sigma}x_k$, and defining $\check{A}_{cl,\sigma}=\check{A}_\sigma -\sum_{j=1}^{m}r_j^{-1}\check{B}_{j\sigma}B_{j\sigma}^TP_\sigma$, the discrete closed-loop system becomes $x_{k+1}=\check{A}_{cl,\sigma}x_k+\check{C}_\sigma+d_{\sigma}$.

Now, consider the Lyapunov function
\begin{align} \label{PW_nonmono_lyapunov}
V(x_k,\hat P)&=\begin{bmatrix}
x_k \\ x_{k+1}
\end{bmatrix}^T \hat P\begin{bmatrix}
x_k \\ x_{k+1}
\end{bmatrix}
\end{align}
characterized by $\hat P \in \mathscr{F},$ where
\begin{align*}
\mathscr{F}=\{\hat P \in \mathbb{R}^{2n\times 2n} | 0 \leq \hat P \leq I, V(x_{k+1},\hat P)-V(x_k,\hat P)<0, \forall x_k \in \bar D\backslash \{0\}, d_\sigma  \in \Delta_\sigma\}.
\end{align*}
 The structure of the Lyapunov function is suggested by \cite{chen2020learning} that employs a piecewise quadratic function to parameterize the Lyapunov function. This approach combines the non-monotonic Lyapunov function \cite{ahmadi2008non} and finite-step Lyapunov function \cite{bobiti2016sampling,aeyels1998new} techniques to provide a guarantee by looking at the next few steps. It should be noted that the Lyapunov function may not be necessarily decreasing within any single step, while it must be decreasing within the finite steps taken into account. 

\paragraph{The Learner:}
To realize a Lyapunov function, one needs a mechanism to look for the appropriate values of $\hat P$ within $\mathscr{F}$. For this purpose, we obtain an over-approximation of $\mathscr{F}$ by considering an only finite number of elements in $(\bar D,\Delta)$. 
Let us first define the increment on the Lyapunov function as
\begin{align*}
\Delta V(x,\hat P)&= V(\check{F}_{d,cl}(x),\hat P)-V(x,\hat P)=\begin{bmatrix}
\check{F}_{d,cl}(x)\\\check{F}^{(2)}_{d,cl}(x)
\end{bmatrix}^T \hat P \begin{bmatrix}
\check{F}_{d,cl}(x)\\\check{F}^{(2)}_{d,cl}(x)
\end{bmatrix} - \begin{bmatrix}
x\\\check{F}_{d,cl}(x)
\end{bmatrix}^T \hat P \begin{bmatrix}
x\\\check{F}_{cl}(x)
\end{bmatrix},
\end{align*}
where $\check{F}^{(2)}_{d,cl}(x)=\check{F}_{d,cl}(\check{F}_{d,cl}(x))$.

Furthermore, assume that the set of $N_s$ number of samples is given as below
\begin{align*}
\mathscr{S}=\{(x,\check{F}_{d,cl}(x),\check{F}_{d,cl}^{(2)}(x))_1,\dots,(x,\check{F}_{d,cl}(x),\check{F}_{d,cl}^{(2)}(x))_{N_s}\}.
\end{align*}
Note that $\mathscr{S}$ implicitly includes samples of the disturbance input and the state. 

Now, using $\mathscr{S}$ we obtain the over-approximation
\begin{align*}
\tilde{\mathscr{F}}=\{\hat P \in \mathbb{R}^{2n\times 2n} | 0 \leq \hat P \leq I, \Delta V(x,\hat P)\leq 0 , \forall x \in \mathscr{S},d_\sigma  \in \Delta_\sigma\}.
\end{align*}

To find an element in $\tilde{\mathscr{F}}$, there exist efficient iterative techniques that are well-known as cutting-plane approaches. See e.g. \cite{atkinson1995cutting,elzinga1975central,boyd2007localization}. In \cite{chen2020learning}, the analytic center cutting-plane method (ACCPM) \citep{goffin1993computation,nesterov1995cutting,boyd2004convex} is employed in an optimization problem:
\begin{align} \label{PW_ACCPM}
\hat P^{(i)}=\underset{\hat P}{\arg \min} \quad -\sum_{x \in \mathscr{S}_i} \log(-\Delta V(x,\hat P))-\log \det(I-\hat P)-\log \det(\hat P)
\end{align}
where $i$ is the iteration index. If feasible, the log-barrier function in the first term guarantees the solution within $\tilde{\mathscr{F}}$ for which the negativity of the Lyapunov difference holds. The other two terms ensure $0 \leq \hat P^{(i)} \leq I$. The solution gives a Lyapunov function 
$V$ based on the set of the samples $\mathscr{S}_i$ in the $i$th stage. On the other hand, if a solution does not exist, the set ${\mathscr{F}}$ is concluded to be empty.

\paragraph{The Verifier:}
The Lyapunov function candidate suggested by (\ref{PW_ACCPM}) may not guarantee asymptotic stability for all $ x \in \bar D$ and $d_\sigma  \in \Delta_\sigma$ since only the sampled space was considered. Therefore, in the next step, we need to verify the Lyapunov function candidate for the uncertain system. To do so, a mixed-integer quadratic program is solved based on the convex hull formulation of the PWA:
\begin{align}\label{PW_verify1}
\underset{\mathrm{x}^j,\mathrm{u}^j,\mathrm{d}^j,\mu^j}{\max} \qquad \quad  &\begin{bmatrix}
\mathrm{x}^1\\\mathrm{x}^2
\end{bmatrix}^T \hat P^{(i)} \begin{bmatrix}
\mathrm{x}^1\\\mathrm{x}^2
\end{bmatrix} - \begin{bmatrix}
\mathrm{x}^0\\\mathrm{x}^1
\end{bmatrix}^T \hat P^{(i)}  \begin{bmatrix}
\mathrm{x}^0\\\mathrm{x}^1
\end{bmatrix} \\
\text{subject to} \nonumber\\
&Z_{\bar D} \mathrm{x}^0 \leq z_{\bar D}, \lVert \mathrm{x}^0\rVert_\infty \geq \epsilon \label{PW_verify3}\\
&\quad \mathrm{u}^j=\omega(\mathrm{x}^j)\label{PW_verify4}\\
&\quad Z_\sigma \mathrm{x}^j_\sigma \leq \mu^j_\sigma z_\sigma, Z_u \mathrm{u}_\sigma \leq \mu^j_\sigma z_u, |\mathrm{d}^j_{\sigma i}| \leq \mu^j_\sigma \bar{d}_{\sigma i},\label{PW_verify6}\\
&\quad (1,\mathrm{x}^j,\mathrm{u}^j,\mathrm{d}^j,\mathrm{x}^{j+1})= \sum_{\sigma=1}^{N_\sigma}(\mu^j_\sigma,\mathrm{x}^j_\sigma,\mathrm{u}^j_\sigma,\mathrm{d}^j_\sigma,A_\sigma \mathrm{x}^j_\sigma+B_\sigma \mathrm{u}^j_\sigma+\mu^j_\sigma c_\sigma+\mathrm{d}^j_\sigma)\label{PW_verify7}\\
&\quad \mu_\sigma \in \{0,1\}, \forall\sigma \in \{1,\dots,N_\sigma\}, i \in \{1,\dots,n\},j \in \{0,1\}, \label{PW_verify8} \end{align}
where a ball of radius $\epsilon$ around the origin is excluded from the set of states, and $\epsilon$ is chosen small enough in (\ref{PW_verify3}). This is due to Remark \ref{PW_remark_local_model_limitation} in Appendix \ref{app:stability} and the fact that the numerical value of the objective becomes considerably small when approaching the origin. This makes the negativity of the objective too hard to verify around the origin. For more details in the implementation of the algorithm, we refer the reader to \cite{chen2020learning}.

 The system is given by (\ref{PW_verify7}) and (\ref{PW_verify8}). To define the piecewise system in a mixed-integer problem, similar to \cite{chen2020learning}, we use the convex-hull formulation of piecewise model that is presented in \cite{marcucci2019mixed}. However, to consider the uncertainty, we compose a slightly different system where we define extra variables to model the disturbance input.

Constraints (\ref{PW_verify3}), and (\ref{PW_verify6}) define the sets of the initial condition, the state, the control, and the disturbance inputs, respectively. Furthermore, the feedback control is implemented by (\ref{PW_verify4}).

To certify the closed-loop system as asymptotically stable, the optimal value returned by the MIQP (\ref{PW_verify1}) is required to be negative. Otherwise, the argument $({\mathrm{x}^0}^*,{\mathrm{x}^1}^*,{\mathrm{x}^2}^*)$ of the optimal solution is added to the set of samples ${\mathscr{S}}$ as a counter-example.

\subsection{Stability Analysis}

Combining the uncertainty bounds in Section \ref{sec4} and the Lyapunov-based verification results of this section, we are able to prove the following practical stability results of the closed-loop system. 

\begin{theorem}\label{thm:stability}
Suppose that the MIQP (\ref{PW_verify1}) yields a negative optimal value. Let $B_\epsilon$ denote the set $\{x\in\mathbb{R}^n | \rVert x \rVert_\infty \leq \epsilon\}$, i.e., the ball of radius $\epsilon$ in infinity norm around the origin. 
Then the set $B_\epsilon$ is asymptotically stable for the closed-loop system  (\ref{PW_PW_dis_lin_cl}). The largest sub-level set of $V$, i.e., $\{x\in\mathbb{R}^n\,|\,V(x)\le c\}$ for some $c$, contained in $\bar D$ is a verified under-approximation of the real ROA. 
\end{theorem}

The proof can be found in Appendix \ref{app:stability}. Remark \ref{PW_remark_local_model_limitation} in Appendix \ref{app:stability} also discusses how to bridge the gap between convergence to $B_\epsilon$ and the convergence to the origin. 

\section{Numerical Results} \label{sec6}

To validate the proposed piecewise learning and verification  technique we implemented the approach on the pendulum system as (\ref{SOL_pend}) and the dynamical vehicle system \cite{pepy2006path}. Moreover, we compared the results with other techniques presented in the literature. To make a fair comparison, we have taken the parameters of the system from \cite{chang2020neural}. We performed all the simulations in Python 3.7 on a 2.6 GHz Intel Core i5 CPU.

\subsection{Pendulum System}
 For the pendulum system, we discuss the simulation results in three sections. In the first section, we will explain the procedure of identifying the uncertain PWA model with a piecewise feedback control. In the second section, we verify the closed-loop uncertain system and obtain an ROA in $\bar D$. In the third section, we will present the comparison results.   
\subsubsection{Identify and Control}
Control objective is to stabilize the pendulum at the top equilibrium point given by $x_{\text{eq}}=(0,0)$. First, we start with learning a piecewise model together with the uncertainty bounds, and the feedback control. For this purpose, we sample the system, and update our model as discussed in section \ref{PW_section:PW_identify}. We set the sampling time as $h=5 $ms. Accordingly, the value function and the control rule are updated online as in section \ref{PW_section:PW_control}. Then, to verify  the value to be decreasing within each mode, it only remains to calculate the uncertainty bounds using the results obtained in section \ref{sec4}.

To make a visualization of the nonlinearity in the pendulum system (\ref{SOL_pend}) possible, we portray the second dynamic assuming $u=0$ in Fig. \ref{fig:dyn_pend}, where the first dynamic is only linear. The procedure of learning is illustrated through several stages in Fig. \ref{Fig:learning_model}. In the first column from the left, we illustrated the estimations only for the second dynamic with $u=0$ to be comparable to Fig. \ref{fig:dyn_pend}. Accordingly, it can be observed that the system identifier is able to closely approximate the nonlinearity with a piecewise model. More details on the uncertainty bounds obtained are provided in Appendix \ref{app:numerical}.

\begin{figure}[htbp]
	\floatconts
	{Fig:learning_model}
	{\caption{ The procedure for learning the dynamics by the PWA model is illustrate step-by-step that shows the convergence of the identifier. Subfigure (g) illustrates the second dynamic of pendulum system (\ref{SOL_pend}) assuming $u=0$ that is $f_2(x_1,x_2)$. Subfigures (a)-(f) show the improvement of estimations of $f_2(x_1,x_2)$ given by (g), as the number of samples increases.} }
	{%
	\begin{minipage}{.6\textwidth}
		\subfigure[]{\label{fig:dyn_pend0}%
			\includegraphics[width=0.3\linewidth]{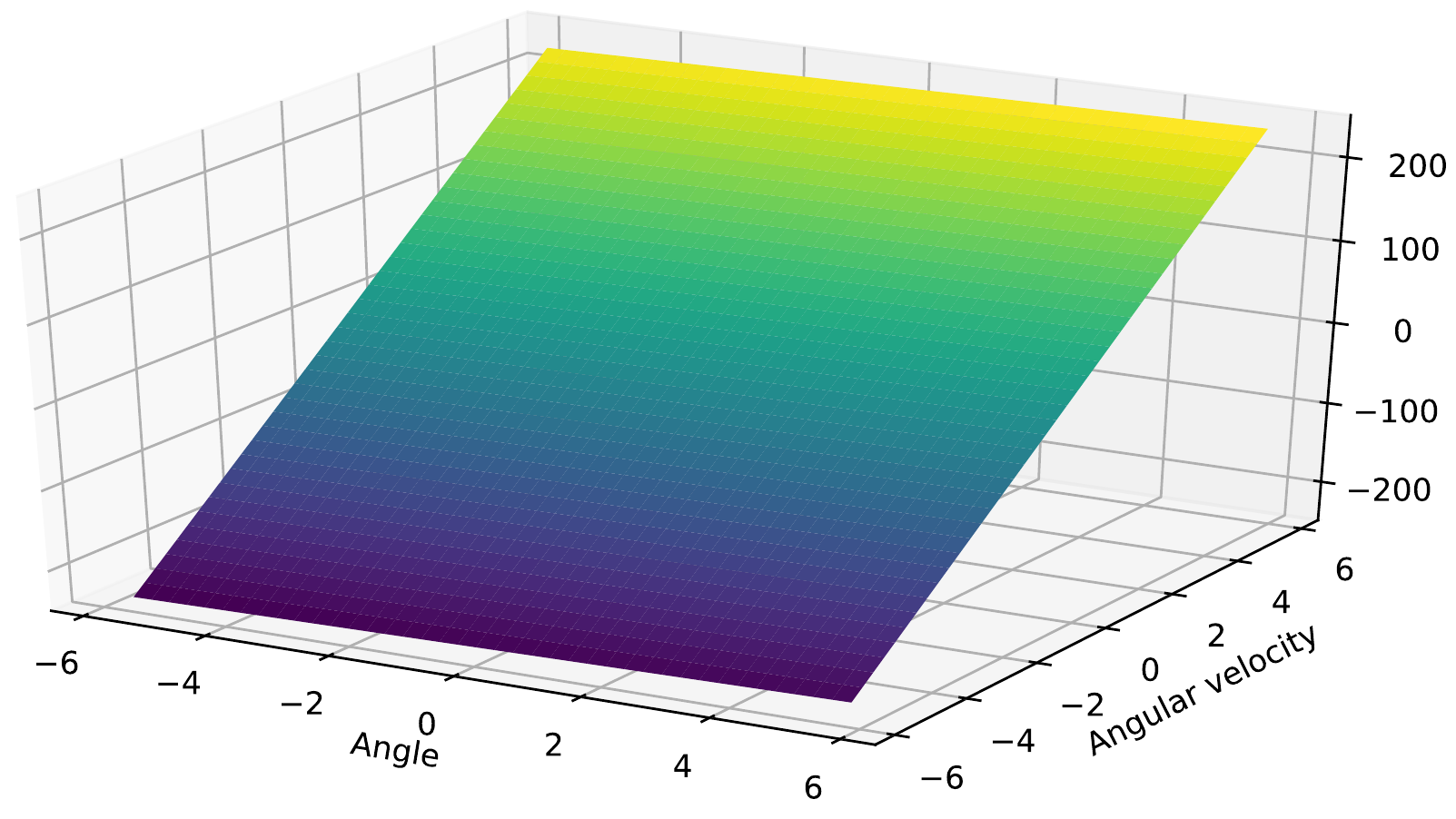}}%
		\quad
		\subfigure[]{\label{fig:dyn_pend1}%
			\includegraphics[width=0.3\linewidth]{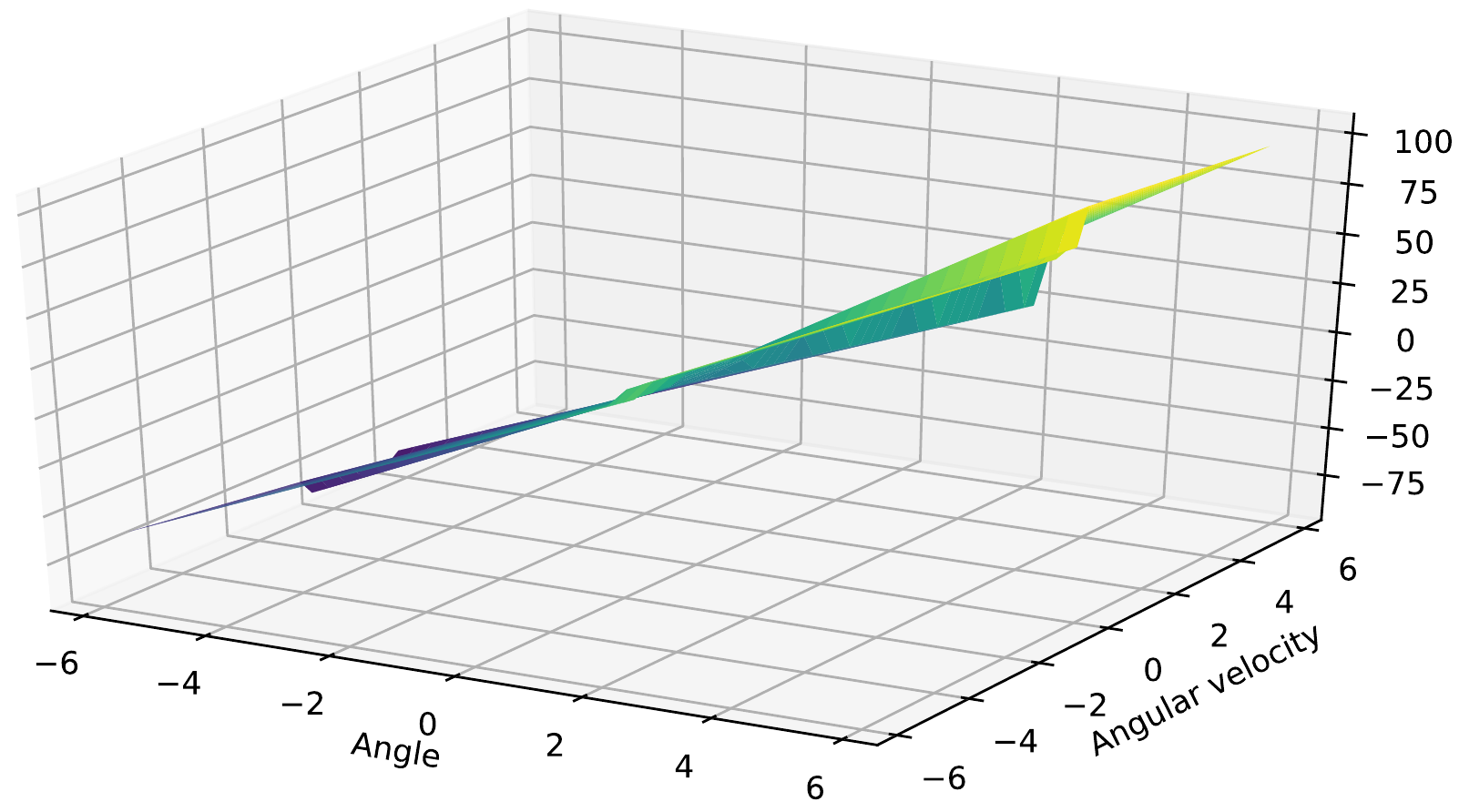}}
		\quad
		\subfigure[]{\label{fig:dyn_pend2}%
			\includegraphics[width=0.3\linewidth]{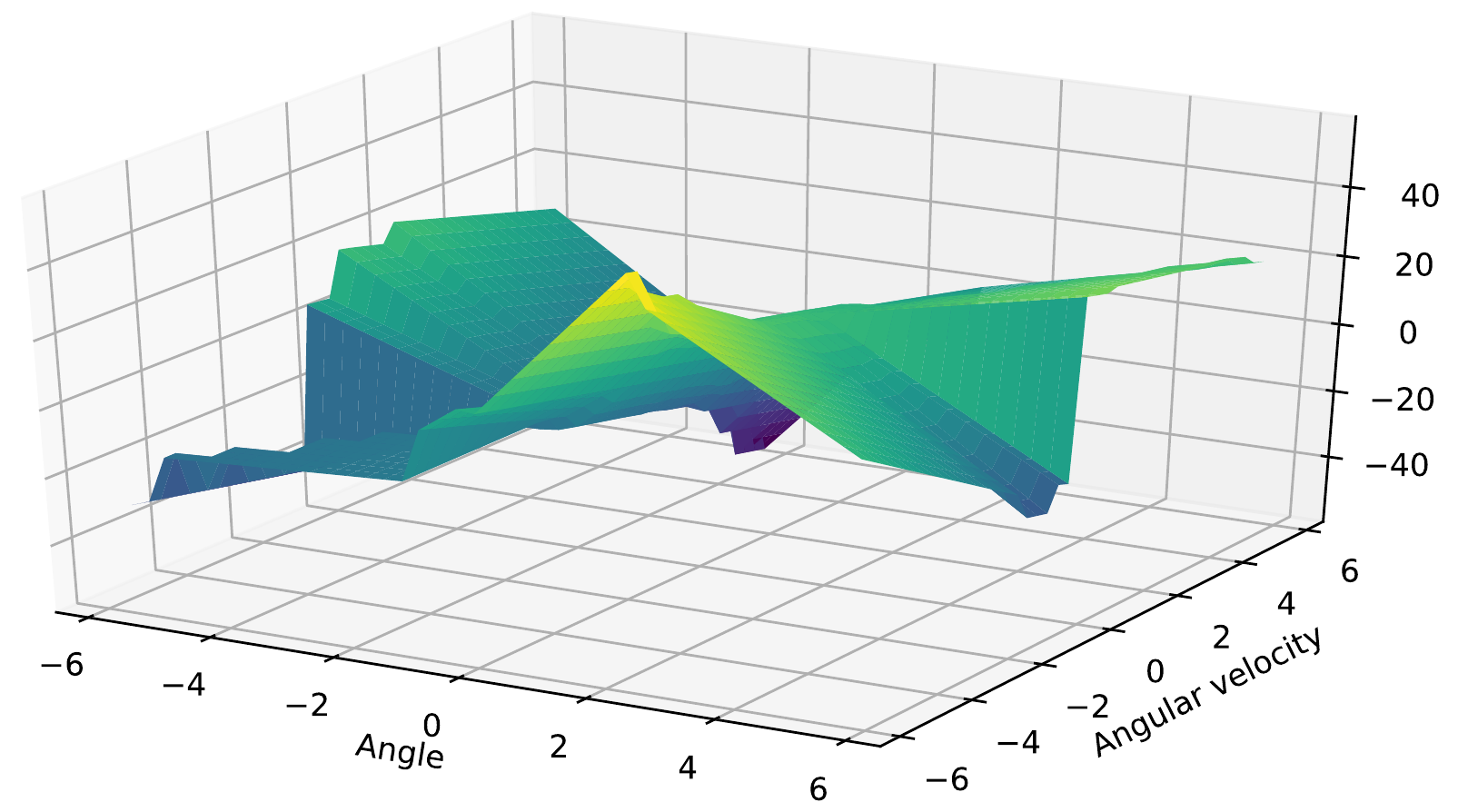}}
		\newline 
		\subfigure[]{\label{fig:dyn_pend3}%
			\includegraphics[width=0.3\linewidth]{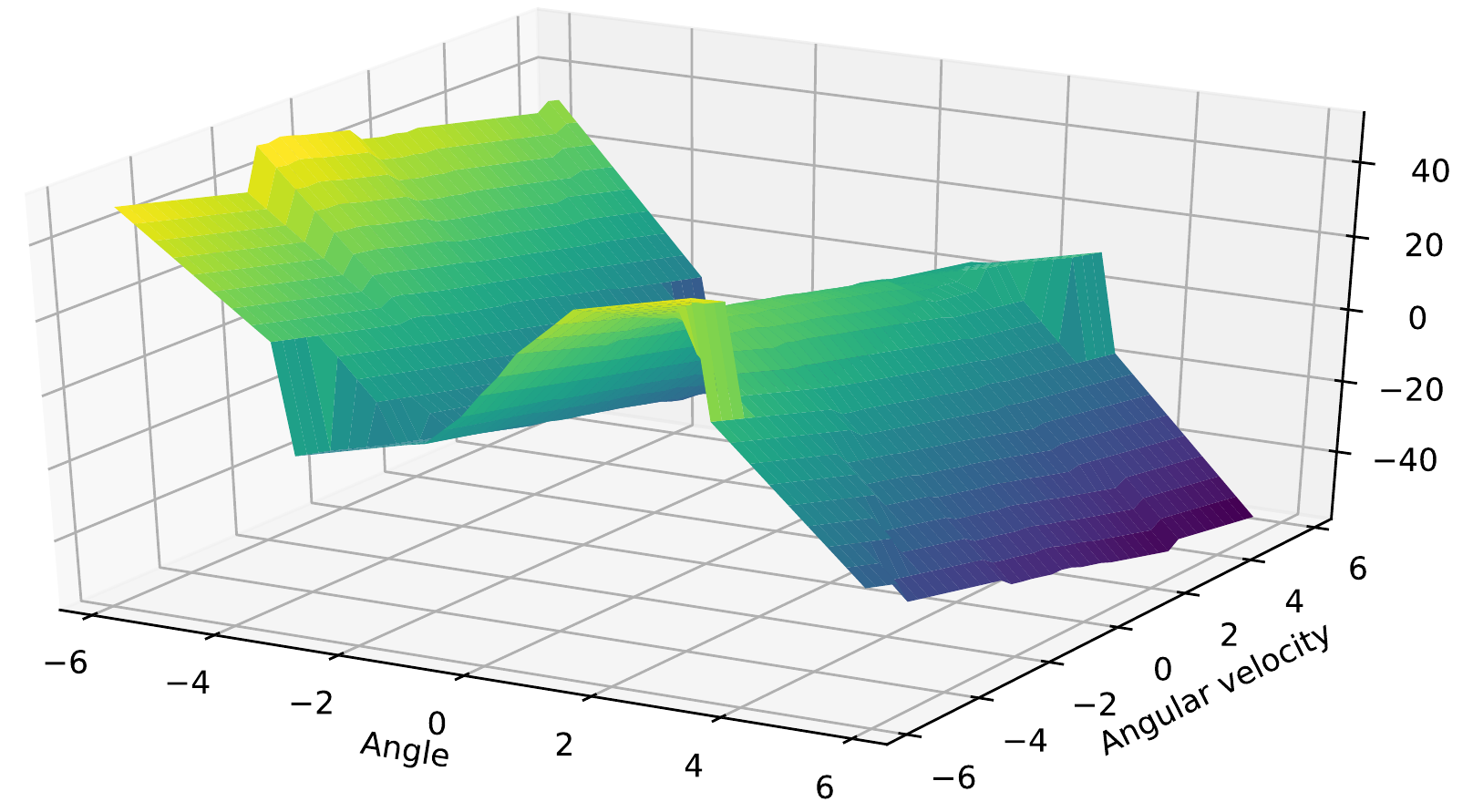}}%
		\quad
		\subfigure[]{\label{fig:dyn_pend4}%
			\includegraphics[width=0.3\linewidth]{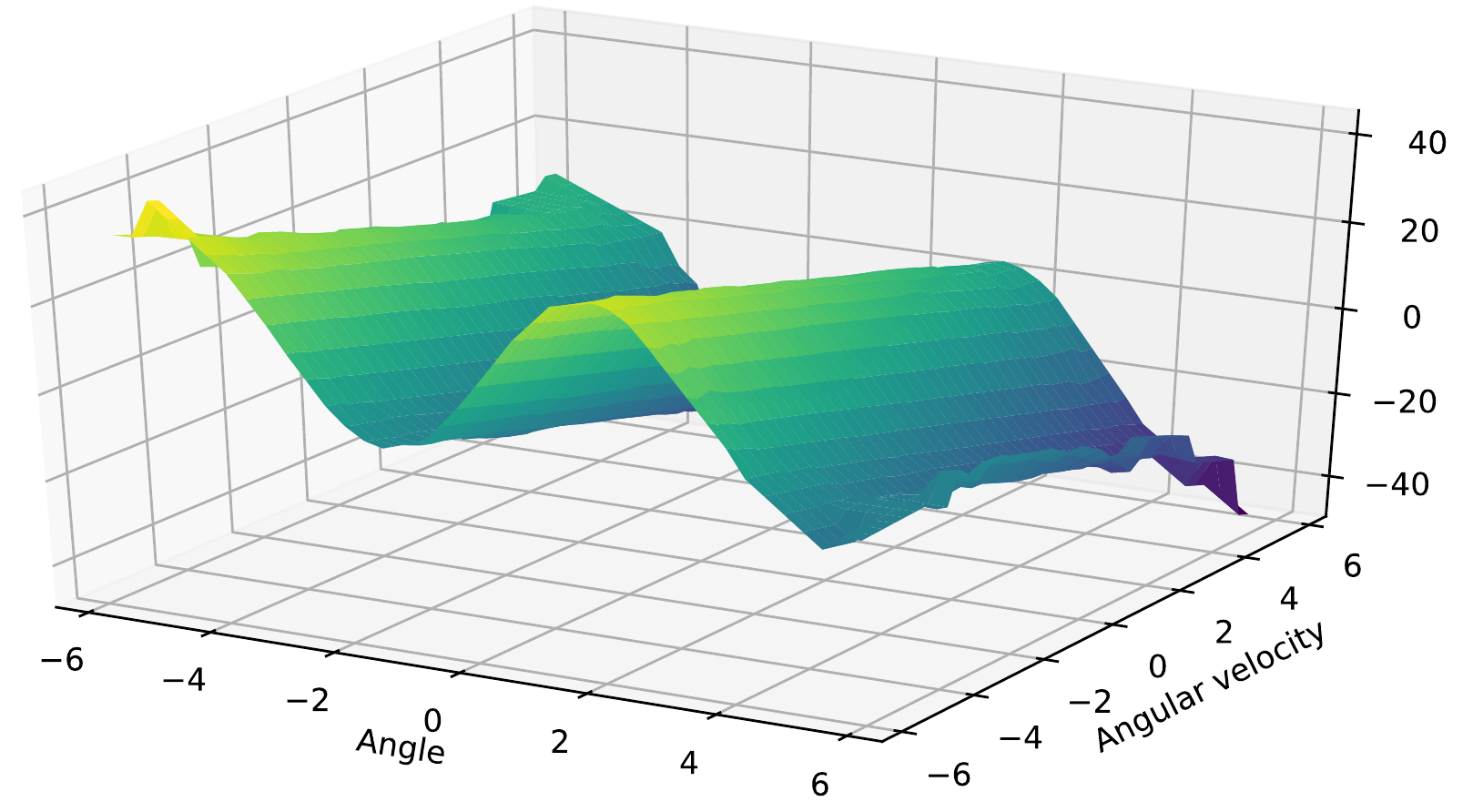}}
		\quad
		\subfigure[]{\label{fig:dyn_pend5}%
			\includegraphics[width=0.3\linewidth]{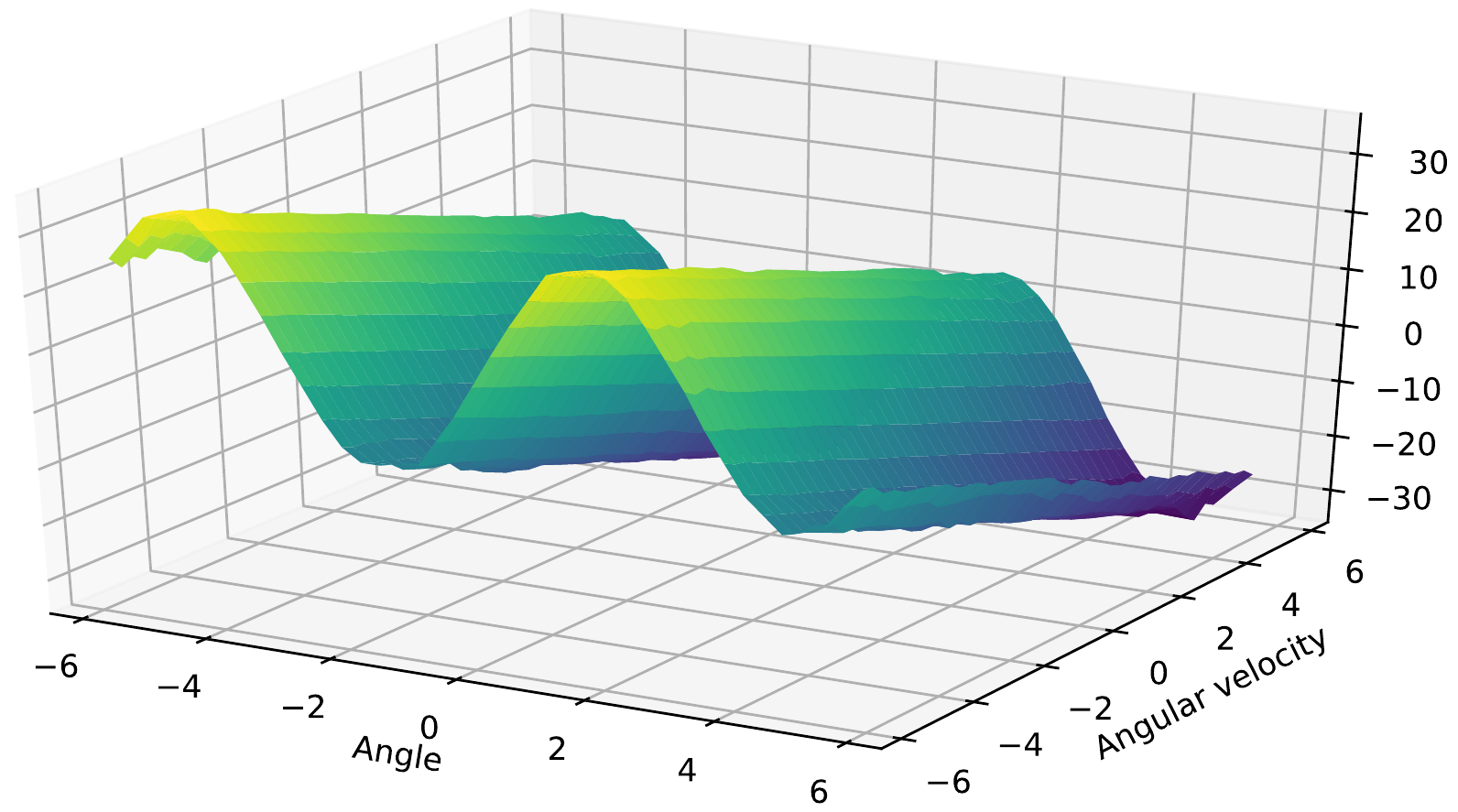}}
	 \end{minipage}%
 \noindent\fcolorbox{gray}{lightgray}{%
 	\begin{minipage}{\dimexpr0.3\textwidth-2\fboxrule-2\fboxsep\relax}
 				\subfigure[]{\label{fig:dyn_pend}%
 			\includegraphics[width=1\linewidth]{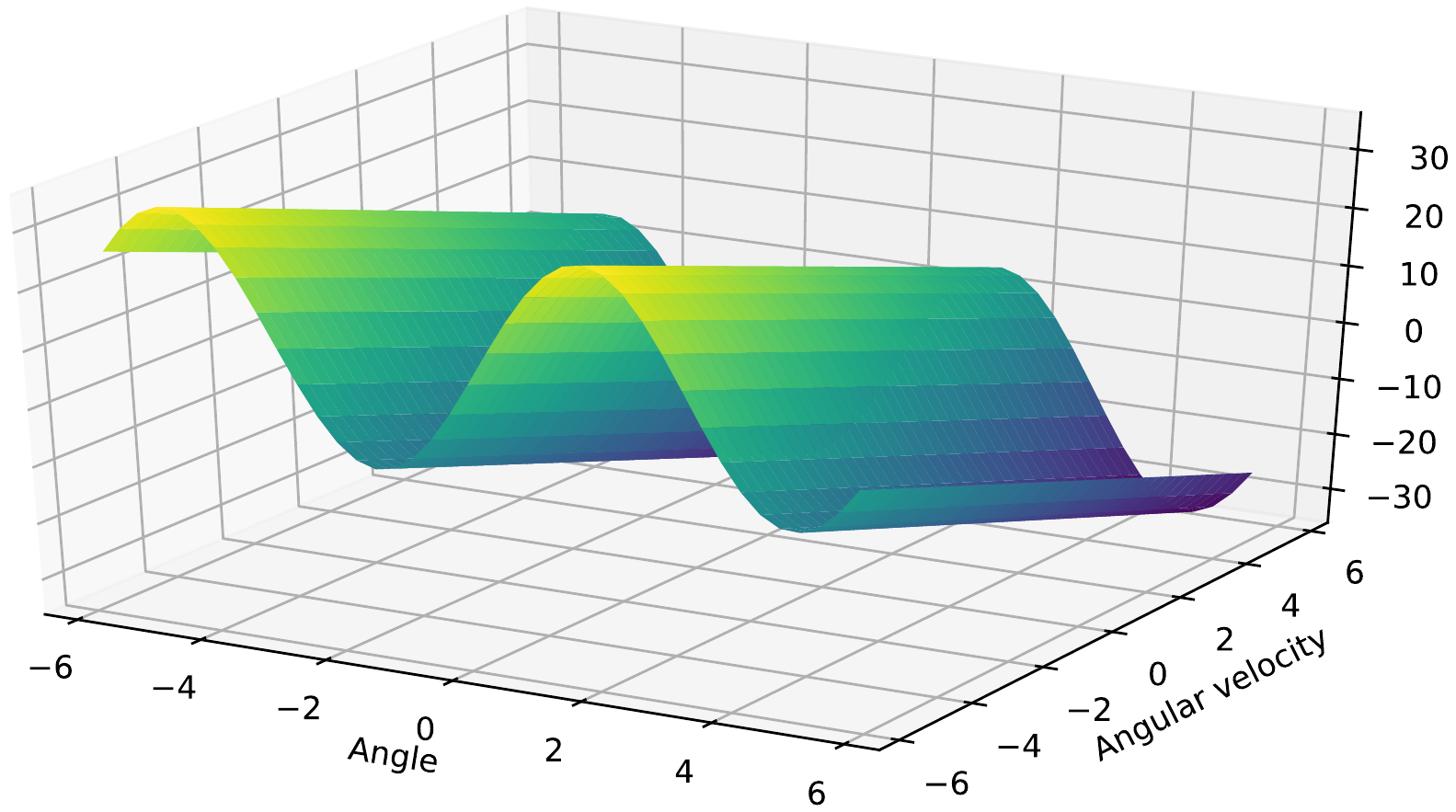}}%
 \end{minipage}}%
	}
	\vspace{-10mm}
\end{figure}

\subsubsection{Verification}
Having the system identified and the feedback control, we can apply the verification algorithm based on MIQP problem. As done in \cite{chen2020learning}, we implemented the learner in CVXpy \cite{diamond2016cvxpy} with MOSEK \cite{mosek2020reference} solver, and the verifier in Gurobi 9.1.2 \cite{gurobi2020reference}.  

We choose $\bar D$ such that $x_1$ and $x_2 \in [-6,6]$. To verify the system, we ran the algorithm and obtained a matrix $\hat{P}$ (who numerical values are given in Appendix \ref{app:numerical}) that characterizes the Lyapunov function as in (\ref{PW_nonmono_lyapunov}).

The largest level set of the associated Lyapunov function in $\bar D$ is pictured in Fig. \ref{PW_fig:ROA} as the the ROA of the closed-loop system. Moreover, we illustrate different trajectories of the controlled system that confirms the verified Lyapunov function by constructing an ROA around the origin. 

\begin{figure}[htbp]
	\floatconts
	{Fig:ROA}
	{\caption{ (a). The obtained ROA of the closed-loop PWA system is illustrated for $x_1$ and $x_2 \in [-6,6]$. The uniform grids denote the modes of the PWA system. Multiple trajectories of the system are shown in a phase portrait where colormap represents the magnitude in the vector field. \\ (b). The comparison results for ROA of the closed-loop system is illustrated for $x_1$ and $x_2 \in [-6,6]$, together with the trajectories of the system. The comparison results for LQR, NN, and SOS are taken from \cite{chang2020neural}. }}
	{\hspace{-5cm}\begin{minipage}{.5\textwidth}
			\centering{
			\subfigure[]{\label{PW_fig:ROA}%
				\includegraphics[width=0.82\linewidth]{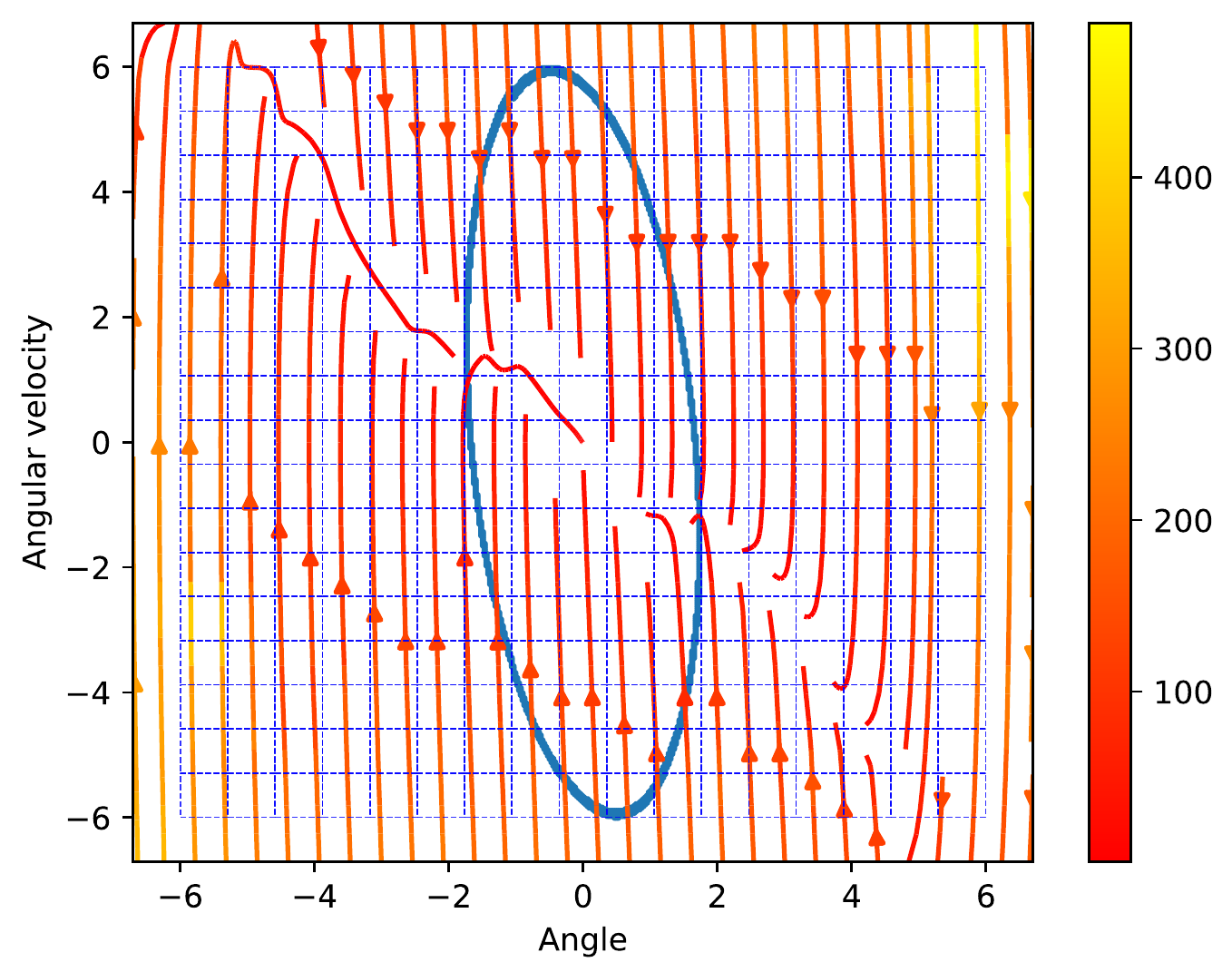}
				}%
			\subfigure[]{\label{PW_fig:ROA_comp}%
				\includegraphics[width=0.85\linewidth]{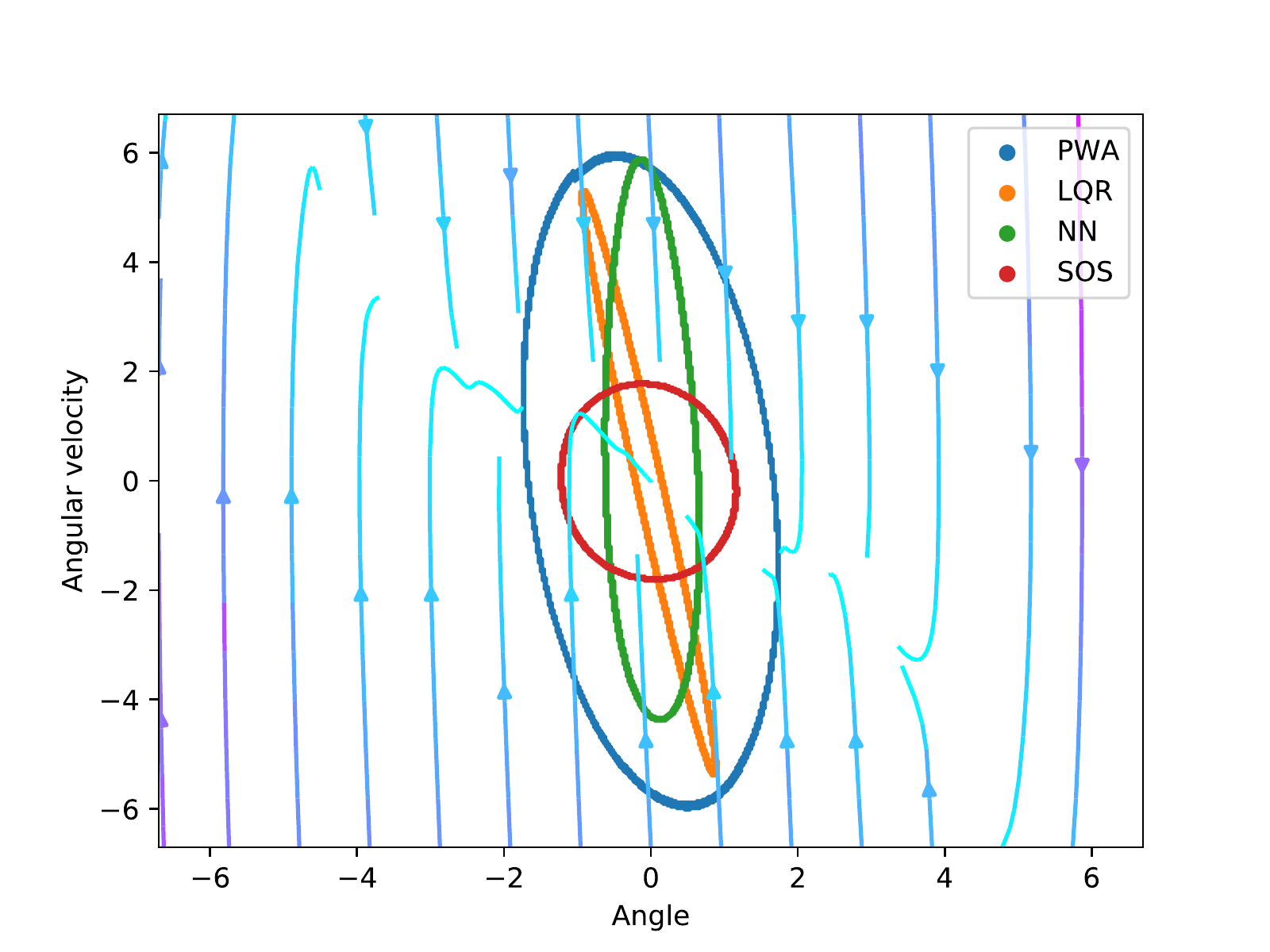}
				}}
				\vspace{-10mm}
		\end{minipage}%
	}
	\vspace{-7mm}
\end{figure}

\subsubsection{Comparison Results}
To highlight the merits of the proposed piecewise learning approach, we compare the ROA obtained by different approaches in the literature. \cite{chang2020neural} proposed a neural network (NN) Lyapunov function for stability verification. According to \cite{chang2020neural}, the comparison done on the pendulum system with linear quadratic regulator (LQR) and sums-of-squares (SOS) showed noticeable superiority of the NN-based Lyapunov approach. Following the comparison results from \cite{chang2020neural}, we compare the ROA obtained by our approach with NN, SOS, and LQR techniques in Fig. \ref{PW_fig:ROA_comp}. Clearly, the ROA obtained by the piecewise controller with the non-monotonic Lyapunov function is considerably larger than the ones obtained by NN, SOS, and LQR algorithms as shown in \cite{chang2020neural}. 

\subsection{Dynamic Vehicle System with Skidding}

We have also implemented the proposed approach in a more complex dynamic vehicle system with skidding, which shows promising results. The results are omitted in the main paper due to the space limit. They can be found in Appendix \ref{app:numerical}.

\section{Conclusion}
For regulating nonlinear systems with uncertain dynamics, a piecewise nonlinear affine framework was proposed in which each piece is responsible for learning and controlling over a partition of the domain locally. Then, in a particular case of the proposed framework, we focused on learning in the form of the well-known {PWA} systems, for which we presented an optimization-based verification approach that takes into account the estimated uncertainty bounds. We used the pendulum system as a benchmark example for the numerical results. Accordingly, an {ROA} resulting from the level set of a learned Lyapunov function is obtained. Furthermore, the comparison with other control approaches in the literature illustrates a considerable improvement in the {ROA} using the proposed framework. As another example, we implemented the presented approach on a dynamical vehicle system with considerably higher number of partitions and dimensions. The results demonstrated that the approach can scale efficiently, hence, can be potentially implemented on more complex real-world problems in real-time. 


\section*{Acknowledgments} 

This work is partially supported by the HUST-WUXI Research Institute through a JITRI-Waterloo joint project, the NSERC Canada Research Chairs (CRC) program, an NSERC Discovery Grant, and an Ontario Early Researcher Award (ERA). 

\bibliography{PW}


\setcounter{page}{1}

\appendix

\section{Model Identification}

\subsection{Continuity of the Identified Model}
\label{app:continuity}

Considering that differentiable bases are assumed, the model identified is differentiable within the interior of $\Upsilon_\sigma$ for $\sigma \in \{1,2,\dots,n_\sigma\}$. However, the pieces of the model may not meet in the boundaries of $\Upsilon_\sigma$ where $x \in \Upsilon_\sigma \bigcap \Upsilon_l$ for any $\sigma \neq l$ and $\sigma$, $l \in \{1,2,\dots,n_\sigma\}$.

Based on our knowledge of system (\ref{SOL_sys}) from which we collect samples the continuity holds for the original system. Hence, in theory, if many pieces are chosen, and enough samples are collected, the edges of pieces will converge together to yield a continuous model. However, choosing arbitrarily small pieces is not practical.

There exist different techniques to efficiently choose the partitions on $D$ and best fit a continuous piecewise model, see e.g. \cite{toriello2012fitting,breschi2016piecewise,ferrari2003clustering,amaldi2016discrete,rebennack2020piecewise}. Such techniques usually involve global adjustments of the model weights and the partitions for which the computations can be considerably expensive. Therefore, we choose to locally deal with the gaps among the pieces. This can be done by a post-processing routine performed on the identified model.  

A rather straightforward technique is to define extra partitions in the margins of each $\Upsilon_\sigma$ to fill the gaps among pieces. The weights of the corresponding pieces added can be chosen according to the weights of the adjacent pieces that are given by the identification. This is done in a way that helps to connect all the pieces together to make a continuous piecewise model. Fig. \ref{PW_fig:pieces} illustrates the process of constructing extra partitions for a two-dimensional case, where we choose them to be in triangular shapes. A similar approach can be taken for generalizing to the $n$-dimensional case.

\begin{figure}[ht]
	\begin{center}
		\includegraphics[width=7cm]{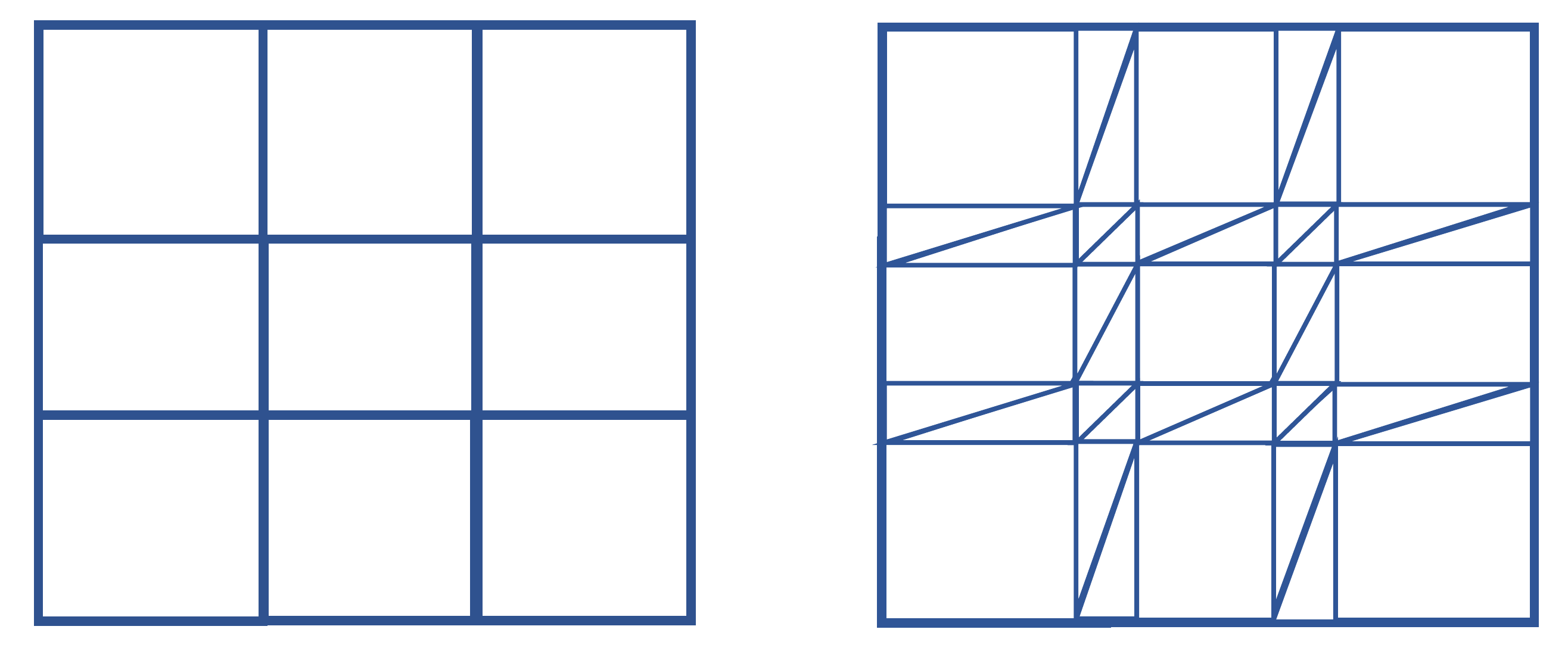}    
		\caption{A scheme for obtaining a continuous piecewise model is illustrated. On the left, partitions on the two-dimensional domain are shown for which the pieces of the model may not be connected in the borders. On the right, some extra triangular pieces are constructed to allow filling the possible gaps in the model.} 
		\label{PW_fig:pieces}
	\end{center}
\end{figure}

\subsection{Database} 

Although an online technique is used to update the piece-wise model along trajectories, we still need to collect a number of samples for each piece of the system. The set of samples recorded will be used later to obtain an estimation of the uncertainty bounds for each mode of the system. For this purpose, we, over time, handpick and save samples that best describe the dynamics in any mode of the piecewise system. 

It should be noted that the database will be processed offline to extract the uncertainty bounds. Hence, it does not affect the online learning procedure and its computational cost. Any sample of the system, to be stored in the database, includes $({\Theta_k}^s,\hat {\dot x}_k)$,  where the state derivative is approximated by $\hat {\dot x}_k =(x_{k}-x_{k-1})/h$ and ${\dot e}_k$. For better results, higher-order approximations of the state derivative can be employed.

Different techniques can be employed to obtain a summary of the samples collected.  We assume a given maximum size of the database $N_d$. Then, for any mode of the piecewise model, we keep adding the samples with larger prediction errors to the database. Therefore, at any step, we compare the prediction error $ {\dot e}_k=\|{\dot x}_k-\hat {\dot x}_k\|$ with the most recent average error $\bar {\dot e}_{k\sigma}$ obtained for the active piece. Hence, if the condition ${\dot e}_k>\eta\bar {\dot e}_{k\sigma}$ holds we add the sample to the database, where the constant $ \eta>0$ adjusts the threshold. If the maximum number of samples in database is reached, we replace the oldest sample with the recent one.

\section{Analysis of Uncertainty Bounds}

\subsection{Proof of Theorem \ref{thm:bound1}}\label{app:thm1}

 \begin{proof}
	According to Assumption \ref{PW_assu_Lip_e}, it is straightforward to show that the prediction error can be bound for any $\sigma$ by using the samples in partition $\sigma$ as 
	\begin{align*}
	|\hat{F}_i(x^s,u^s)-{F}_i(x^s,u^s)|&\leq |\hat{F}_i(x^s,u^s)-\tilde{F}_i(x^s,u^s)|+|\tilde{F}_i(x^s,u^s)-{F}_i(x^s,u^s)| \\
	&\leq|\hat{F}_i(x^s,u^s)-\tilde{F}_i(x^s,u^s)|+\varrho_e |\tilde{F}_i(x^s,u^s)|\\
	&\leq \underset{s \in \Upsilon_\sigma}{\max}(|\hat{F}_i(x^s,u^s)-\tilde{F}_i(x^s,u^s)|+\varrho_e |\tilde{F}_i(x^s,u^s)| )\\
	&=\bar d_{e\sigma i}.
	\end{align*} 
\end{proof}

\subsection{Quadratic Programs for Bounding Errors}\label{app:qperror}

\begin{assumption}\label{PW_assu_Lip_x}
	For system (\ref{SOL_sys}),  $\exists\varrho_x\in\mathbb{R}^{n}_+$ such that we have
	\begin{align*}
	{| F_i(x_0,u)-F_i(y_0,u)|} \leq \varrho_{xi} {\lVert x_0-y_0\rVert},
	\end{align*} 
	for any $x_0,y_0 \in D$, and $u\in\Omega$, where $i \in \{1,\dots,n\}$.
\end{assumption}

\begin{assumption}\label{PW_assu_Lip_u}
	For system (\ref{SOL_sys}), $\exists\varrho_u\in\mathbb{R}^{n}_+$ such that we have
	\begin{align*}
	{| F_i(x,u_0)-F_i(x,w_0)|} \leq \varrho_{ui} {\lVert u_0-w_0\rVert},
	\end{align*} 
	for any $x \in D$, and $u_0,w_0$ $\in\Omega$, where $i \in \{1,\dots,n\}$.
\end{assumption}

\begin{assumption}\label{PW_assu_Lip_estim}
	An initial estimation of $\varrho_e$ and Lipschitz constants $\varrho_{xi}$ and $\varrho_{ui}$ is known. 
\end{assumption}
The following results and the bounds will directly depend on the choice of $\varrho_x$, and $\varrho_u$. However, this is the least we can assume that allows us to carry out the computations. Moreover, making such assumptions is not restrictive in practice since we often have a general knowledge of the application. Moreover, the learning may be first started with an initial guess of the continuity constants. Later, if the samples collected override the assumption made, we can update these values.  

To calculate the uncertainty bound for any piece, we first look for the largest gap existing among the samples within each piece. Fig. \ref{fig:subfigex} illustrates an example of how this gap is affected by the number of samples in a particular mode. To obtain the radius, we solve a quadratic programming (QP) problem for each piece. The solution to the following QP returns the centre $c_{x\sigma}^*$ at which an $n$-dimensional ball of the largest radius  $r_{x\sigma}^*$ can be found in the $\sigma$th piece such that no samples $x^s$ are contained in this ball: 
\begin{align}\label{PW_gap_center_x}
&\underset{c_{x\sigma},r_{x\sigma}}{\arg\max} \qquad \quad r_{x\sigma}\\ \nonumber
&\text{subject to} \qquad{c_{x\sigma}\in \Upsilon_\sigma}\\ \nonumber
&\hspace{2.2cm}\text{for} \quad s\in S_{\Upsilon\sigma}: \quad \lVert x^s-c_{x\sigma}\rVert\geq r_{x\sigma}
\end{align}
Similarly, we can obtain the centre $c_{u\sigma}^*$ and radius $r_{u\sigma}^*$ to represent the sample gap as an $m$-dimensional ball in the control space by solving 
\begin{align}\label{PW_gap_center_u}
&\underset{c_{u\sigma},r_{u\sigma}}{\arg\max} \qquad \quad r_{u\sigma}\\ \nonumber
&\text{subject to} \qquad c_{u\sigma}\in \Omega\\ \nonumber
&\hspace{2.2cm}\text{for} \quad s\in S_{\Upsilon\sigma}: \quad \lVert u^s-c_{u\sigma}\rVert\geq r_{u\sigma}.
\end{align}

\begin{figure}[htbp]
	\floatconts
	{fig:subfigex}
	{\caption{Sub-figures (a)-(e) denote the sample gaps located for different numbers of the samples. It is observed that the radius of the gap decreases by increasing the number of samples.}}
	{%
		\subfigure[]{\label{fig:circle0}%
			\includegraphics[width=0.15\linewidth]{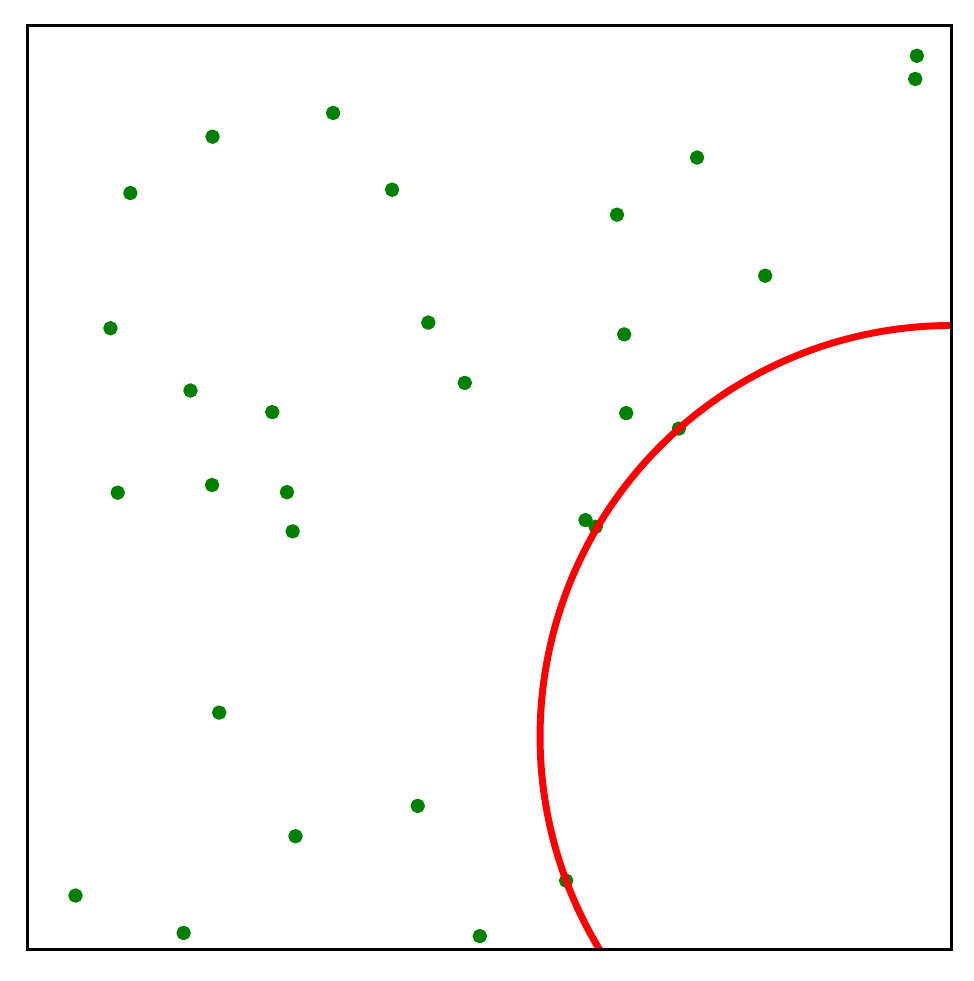}}%
		\quad
		\subfigure[]{\label{fig:square1}%
			\includegraphics[width=0.15\linewidth]{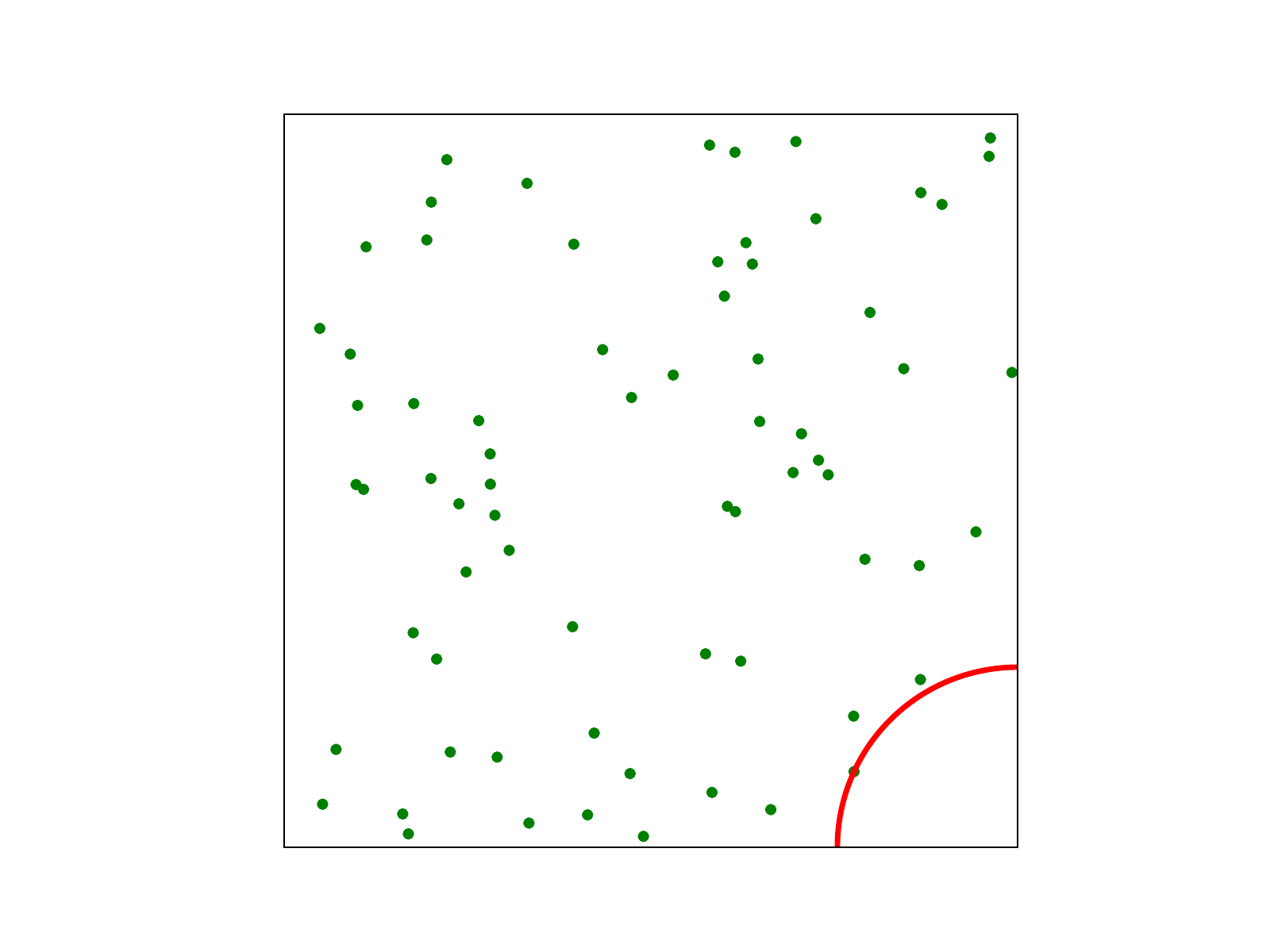}}
		\quad
		\subfigure[]{\label{fig:square2}%
			\includegraphics[width=0.15\linewidth]{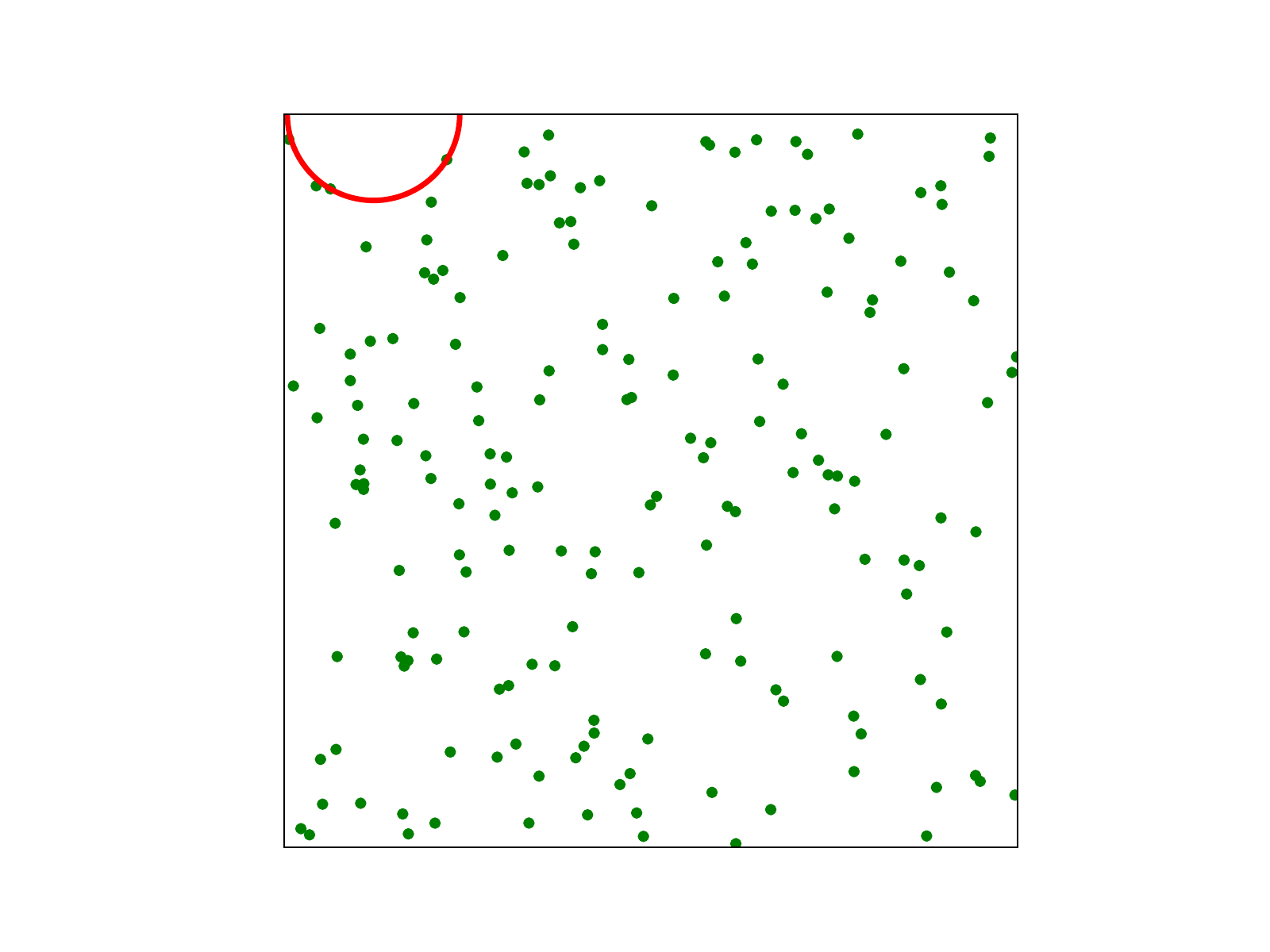}}
		\quad
		\subfigure[]{\label{fig:circle3}%
			\includegraphics[width=0.15\linewidth]{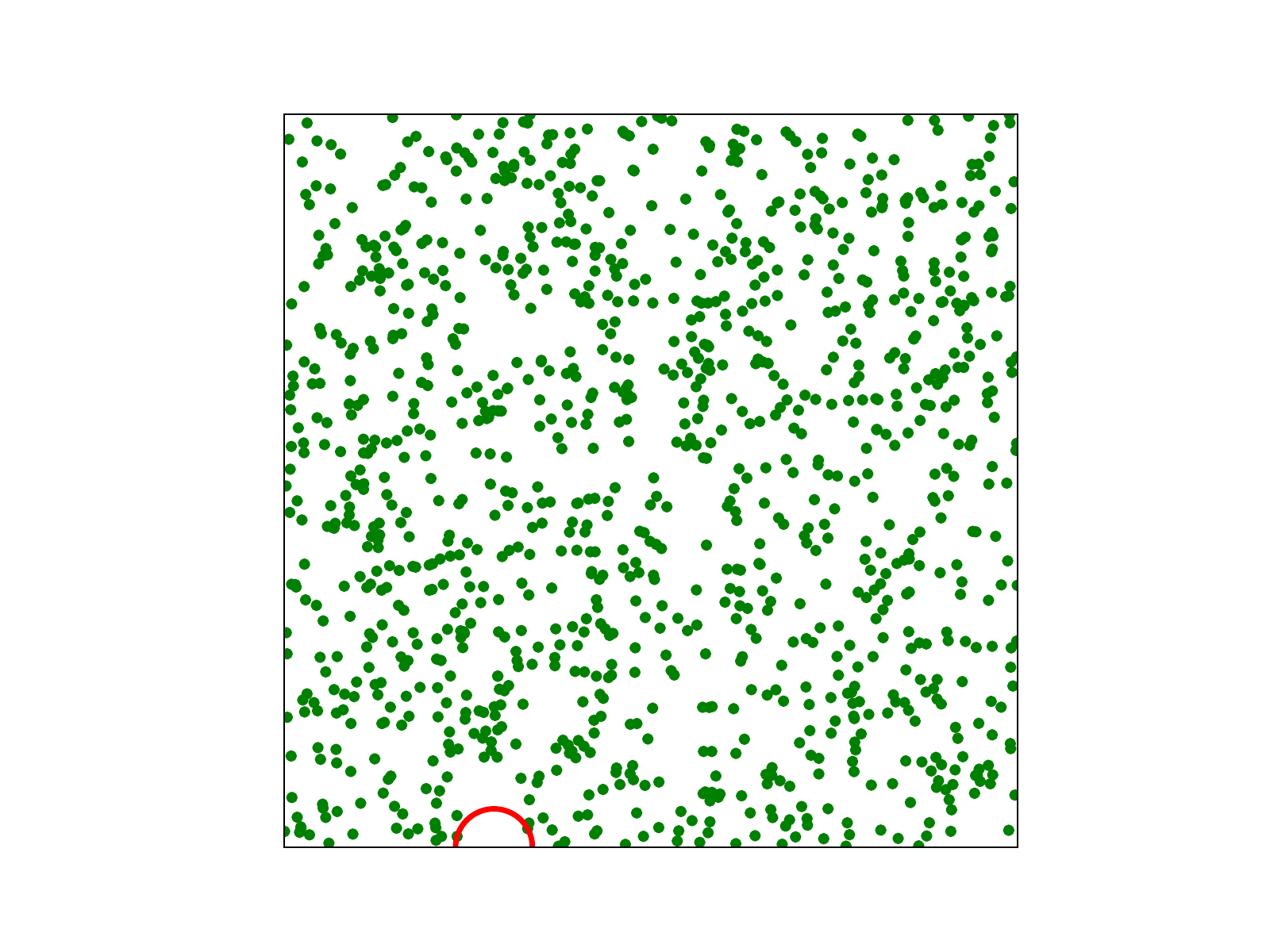}}%
		\quad
		\subfigure[]{\label{fig:square4}%
			\includegraphics[width=0.15\linewidth]{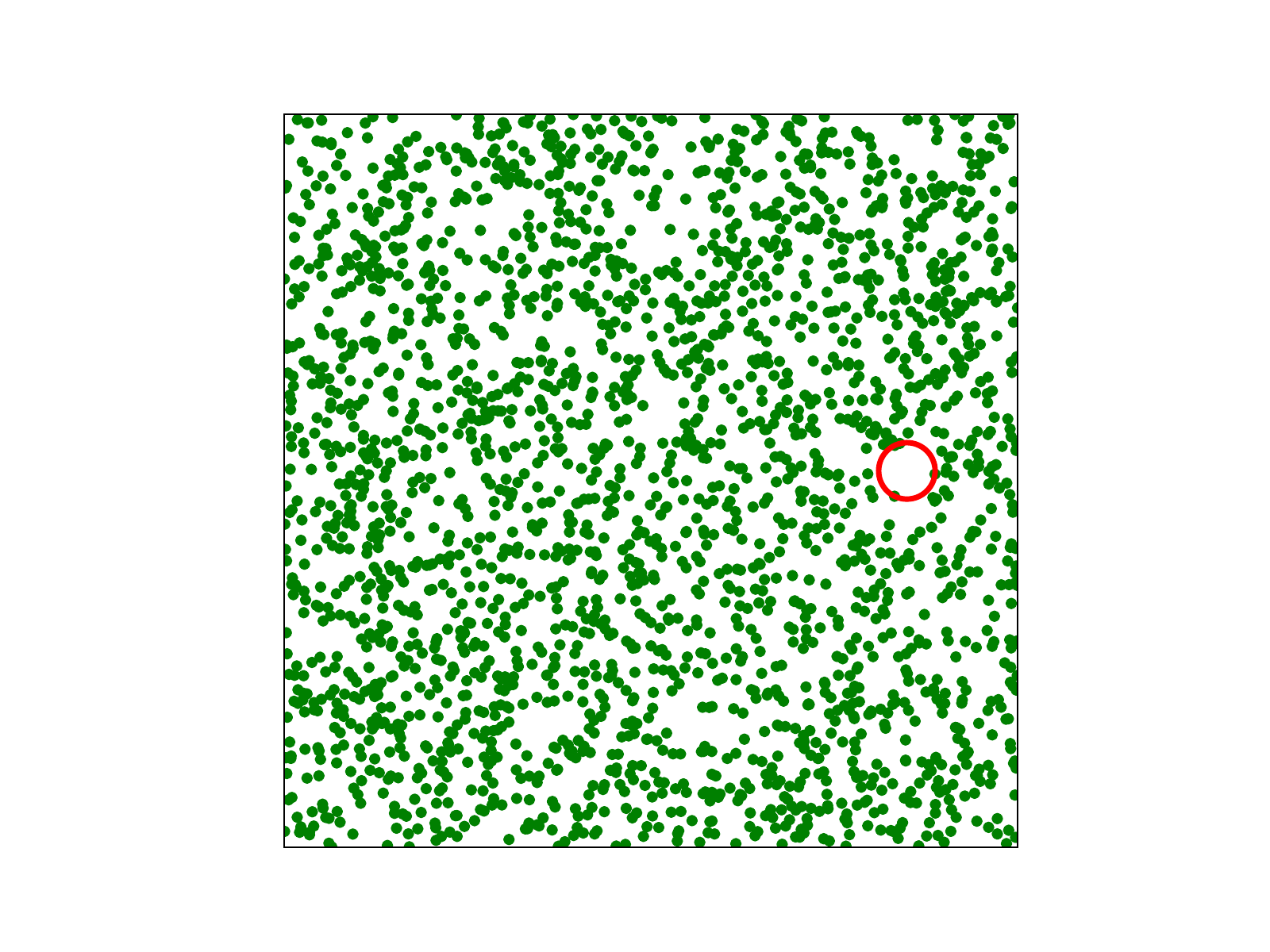}}

	}
\end{figure}

\begin{figure}
\begin{tikzpicture}
\begin{axis}[
xmin=-10, xmax=48,
ymin=-300, ymax=350,
axis lines = left,
xlabel = \(x\),
ylabel = {\(\)},
ticks=none
]
\addplot [
domain=-5:28, 
samples=100, 
color=red,
]
{0.3*x^2 - 2*x - 1};
\addlegendentry{\(F_i\)}

\addplot [
domain=0:32, 
samples=100, 
color=black,
dashed,forget plot,
name path=f0
]
{15*x-170};

\addplot [
domain=0:32, 
samples=100, 
color=black,
dashed,
forget plot,
name path=f1
]
{-15*x+218};

\path[name path=axis] (axis cs:0,0) -- (axis cs:1,0);

\addplot [
thick,
color=blue,
fill=blue, 
fill opacity=0.05,
forget plot
]
fill between[
of=f0 and f1,
soft clip={domain=13:28},
];

\addplot [
domain=-5:35, 
samples=100, 
color=blue,
]
{1*x +0};
\addlegendentry{\(\hat F_i\)}

\addplot [
domain=-5:35, 
samples=100, 
color=blue,dotted,
forget plot,
name path=l0
]
{1*x + -230};

\addplot [
domain=-5:35, 
samples=100, 
color=blue,dotted,
forget plot,
name path=l1
]
{1*x + 230};

\addplot [
thick,
color=yellow,
fill=yellow, 
fill opacity=0.05,
]
fill between[
of=l0 and l1,
soft clip={domain=-5:35},
];
\addlegendentry{\({\bar{d}}_{\sigma  i}\)}

\addplot[only marks,samples at={-3, -2, 1.5, 3, 6}]{0.3*x^2 - 2*x - 1};
\addplot[only marks,nodes near coords={$x^{s*}$},every node near coord/.style={text=black, anchor=east},mark=*,samples at={13}]{0.3*x^2 - 2*x - 1};
\addplot[only marks,nodes near coords={$c^*_{x\sigma}$},mark=square,samples at={28}]{0.3*x^2 - 2*x - 1};
\end{axis}
\end{tikzpicture}
\centering
\caption{The scheme for obtaining uncertainty bound according to the sample gap. Black dots denote the measurements. The dashed lines are plotted according to the Lipschitz continuity properties.}
\label{PW_uncerainty_schme}
\end{figure}
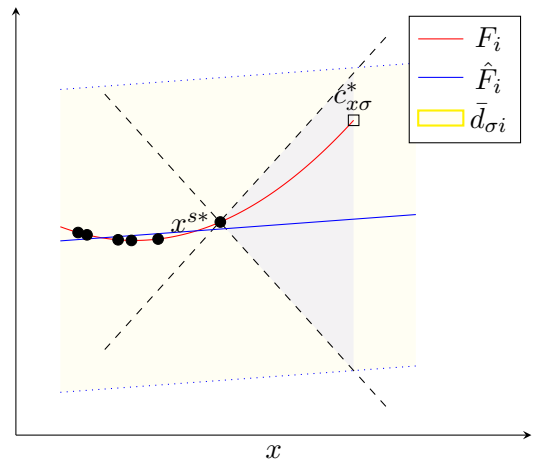

\subsection{Proof of Theorem \ref{thm:bound2}}\label{app:thm2}

\begin{proof}
	According to the Lipschitz condition, the following holds for any $(x,u) \in \Upsilon_\sigma$
	\begin{align}\label{PW_lip_ineq_}
	|{F}_i(x,u)-&{F}_i(x^{s*},u^{s*})|\nonumber\\
	&\leq  |{F}_i(x,u)-{F}_i(x,u^{s*})|+|{F}_i(x,u^{s*})-{F}_i(x^{s*},u^{s*})|\nonumber\\
	&\leq \varrho_{ui} \lVert u-u^{s*} \rVert + \varrho_{xi} \lVert x-x^{s*} \rVert.
	\end{align} 
	Moreover, we have the estimation $\hat F(x,u)$ of the system. Then, the difference is bounded by
	\begin{align*}
	| F_i(x,u)-\hat F_i(x,u)| &\leq |F_i(x,u)-{F}_i(x^{s*},u^{s*})|+|{F}_i(x^{s*},u^{s*})-\hat F_i(x,u)|,\\
	&\leq |F_i(x,u)-{F}_i(x^{s*},u^{s*})|+|{F}_i(x^{s*},u^{s*})-\hat{F}_i(x^{s*},u^{s*})|,\\
	&\quad+|\hat{F}_i(x^{s*},u^{s*})-\hat F_i(x,u)|,\\
	&\leq \varrho_{ui} \lVert u-u^{s*} \rVert + \varrho_{xi} \lVert x-x^{s*} \rVert+ \bar d_{e\sigma i}+|\hat{F}_i(x^{s*},u^{s*})-\hat F_i(x,u)|,\\
	&\leq \varrho_{ui} \lVert u-u^{s*} \rVert + \varrho_{xi} \lVert x-x^{s*} \rVert+ \bar d_{e\sigma i}+|\hat{F}_i(x^{s*},u^{s*})-\hat{F}_i(x^{s*},u)|,\\
	&\quad +|\hat{F}_i(x^{s},u)-\hat F_i(x,u)|,\\
	&\leq \varrho_{ui} \lVert u-u^{s*} \rVert + \varrho_{xi} \lVert x-x^{s*} \rVert+ \bar d_{e\sigma i}+\hat \varrho_{ui} \lVert u-u^{s*} \rVert+\hat \varrho_{xi} \lVert x-x^{s*} \rVert ,
	\end{align*}
	where we used inequality (\ref{PW_lip_ineq_}) and the bound obtained in Theorem \ref{thm:bound1} according to the samples.
	Moreover, considering that the identified model $\hat{F}_i(x,u)$ is known, we can easily compute the corresponding Lipschitz constants $\hat \varrho_{ui}$ and $\hat \varrho_{xi}$. The largest distance with the closest sample $(x^{s*},u^{s*})$ happens in the sample gap given with the radius $r_{x\sigma}^*$, and $r_{u\sigma}^*$. This yields the total bound of the error as
	\begin{align*}
	| F_i(x,u)-\hat F_i(x,u)|\leq \varrho_{ui} r_{u\sigma}^* + \varrho_{xi} r_{x\sigma}^*+ \bar d_{e\sigma i}+\hat \varrho_{ui} r_{u\sigma}^*+\hat \varrho_{xi} r_{x\sigma}^* 
	\end{align*}
\end{proof}

\section{Verification}

\subsection{Searching for a Lyapunov Function } \label{PW_Learning_lyapunov}
  We summarize an altered version of the technique for obtaining a Lyapunov function that is first presented in \cite{chen2020learning} for a deterministic closed-loop system with the neural network controller. Hence, we modify the algorithm to allow the uncertainty together with the feedback control (\ref{PW_PWA_control}).
  
  The procedure includes two stages that are performed iteratively until a Lyapunov function is obtained and verified, or it is concluded that there exists no Lyapunov function in the given set of candidates. 
  
  In the first stage, we assume an initial set of Lyapunov candidates in the form of (\ref{PW_nonmono_lyapunov}). Then, the learner searches for a subset for which the negativity of the Lyapunov difference can be guaranteed with respect to a set of samples collected from the system. If such a subset exists, one element in this subset is proposed as the Lyapunov candidate by the learner.
  
  In the second stage, the proposed Lyapunov candidate is verified on the original system. Noting that the learner only uses a finite number of samples for suggesting a Lyapunov candidate, it may not be valid for all the evolutions of the uncertain system. Accordingly, the verifier either certifies the Lyapunov candidate, or finds a point as the counter-example for which the Lyapunov candidate fails. This sample is added to the set of samples collected from the system. Then, we again proceed to the learner stage with the updated set of samples.
  
  The algorithm is run in a loop, where we start with an empty set of samples in the learner. Then, we continue with proposing a Lyapunov candidate, and adding one counter-example in each iteration of the loop. While growing the set of samples, the set of Lyapunov candidates shrinks in every iteration until one is either validated, or no element is left in the set meaning that no such Lyapunov exists.

\subsection{Convergence of ACCPM}

The convergence and complexity of the ACCPM for searching a quadratic Lyapunov function is discussed in \cite{sun2002analytic,chen2020learning}, where an upper bound is obtained for the number of steps taken until the algorithm exits. 

\begin{lemma}
	Let $\mathscr{F}$ be a convex subset of $\mathbb{R}^{n\times n}$. Moreover, there exists $P_{center} \in \mathbb{R}^{n\times n}$ such that $\{P \in \mathbb{R}^{n\times n}| \lVert P-P_{center}\rVert_F \leq \epsilon\}\subset \mathscr{F}$, where Frobenius norm is used, and $\mathscr{F} \subset \{P \in \mathbb{R}^{n\times n}|0\leq P \leq I\}$. Then, the center cutting-plane algorithm concludes in at most $O(n^3/\epsilon^2)$  steps.
\end{lemma}

\begin{proof}
	See \cite{sun2002analytic,chen2020learning} for the proof. 
\end{proof}

\subsection{Stability Analysis (Proof of Theorem \ref{thm:stability})} \label{app:stability}

\begin{proof}
According to the conditions of the verifier, if the 
optimal value returned by the MIQP (\ref{PW_verify1}) is negative, we have effectively verified the following Lyapunov conditions: 
	\begin{align}
	&V(0)=0, \quad V(x)> 0 ,\quad \forall x \in \bar D\backslash \{0\}, \\
	&V(\check{F}_{d,cl} (x))-V(x)<0, \quad \forall x \in \bar D \backslash B_{\epsilon}, d \in \Delta_\sigma,
	\end{align}
for the uncertain closed-loop system (\ref{PW_PW_dis_lin_cl}). By standard Lyapunov analysis for set stability  \cite{haddad2011nonlinear,jiang2001input}, the set  $B_{\epsilon}$, which is the ball of radius $\epsilon$ in infinity norm around the origin, is asymptotically stable for system (\ref{PW_PW_dis_lin_cl}). Furthermore, any sub-level set of $V(x)$, i.e., $\{x\in\mathbb{R}^n\,|\,V(x)\le c\}$ for some $c$, contained in $\bar D$ is contained in the ROA of $B_{\epsilon}$. 
\end{proof}

\begin{remark} \label{PW_remark_local_model_limitation} Due to the existence of a non-zero additive uncertainty bound, one cannot expect convergence to the origin precisely. This issue is addressed by providing convergence guarantee to a small neighborhood of the origin, i.e., $B_\epsilon$. By collecting enough samples around the origin, a local approximation of the system is obtained by the mode $\sigma=0$ of the identified system, whose domain includes the origin, while $d_\sigma$ can be made arbitrarily small as $x_k \rightarrow 0$. By doing so, we can make $\epsilon$ in Theorem \ref{thm:stability} arbitrarily small and the stability result is practically equivalent to the asymptotic stability of the origin. Alternatively, one can assume that there exists a local stabilizing controller that one can switch to when entering a small neighborhood of the origin. In this case, asymptotic stability can be achieved. 
\end{remark}

\section{Numerical Results}\label{app:numerical}

It should be noted, the learning is started from the mode containing the origin in its domain, that we label by $\sigma=0$. As we collect more random samples in $\Upsilon_0$, we can effectively decrease the uncertainty of the model around the origin, and obtain a local controller. Then, we gradually expand the areas sampled to train the rest of the pieces in the PWA model.

Fig. \ref{Fig:uncertainty} depicts the uncertainty bound obtained over the ROI, where a decreasing magnitude through different stages of learning, i.e. subfigures ($a$) to ($f$), is evident. Fig. \ref{fig:samples} illustrates a map of the pieces on $\bar D$ together with the samples and the sample gaps for each piece as discussed in Appendix \ref{app:qperror}.  

\begin{remark}
	It is worth mentioning that the model obtained and the uncertainty bounds can be further improved by continuing the sampling. In this implementation, we perform sampling only until the uncertainty bound obtained allows us to verify a decreasing value function for each piece of the PWA system.
\end{remark}

\begin{figure}[htbp] 
	\floatconts
	{Fig:uncertainty}
	{\caption{To better illustrate the learning procedure, the step-by-step results of the uncertainty bound corresponding to the results in Fig. \ref{Fig:learning_model} are provided. It is evident that error bound is improved in every step.}}
	{%
		\hspace{13mm}\subfigure[]{\label{Fig:uncertainty0}%
			\includegraphics[width=0.2\linewidth]{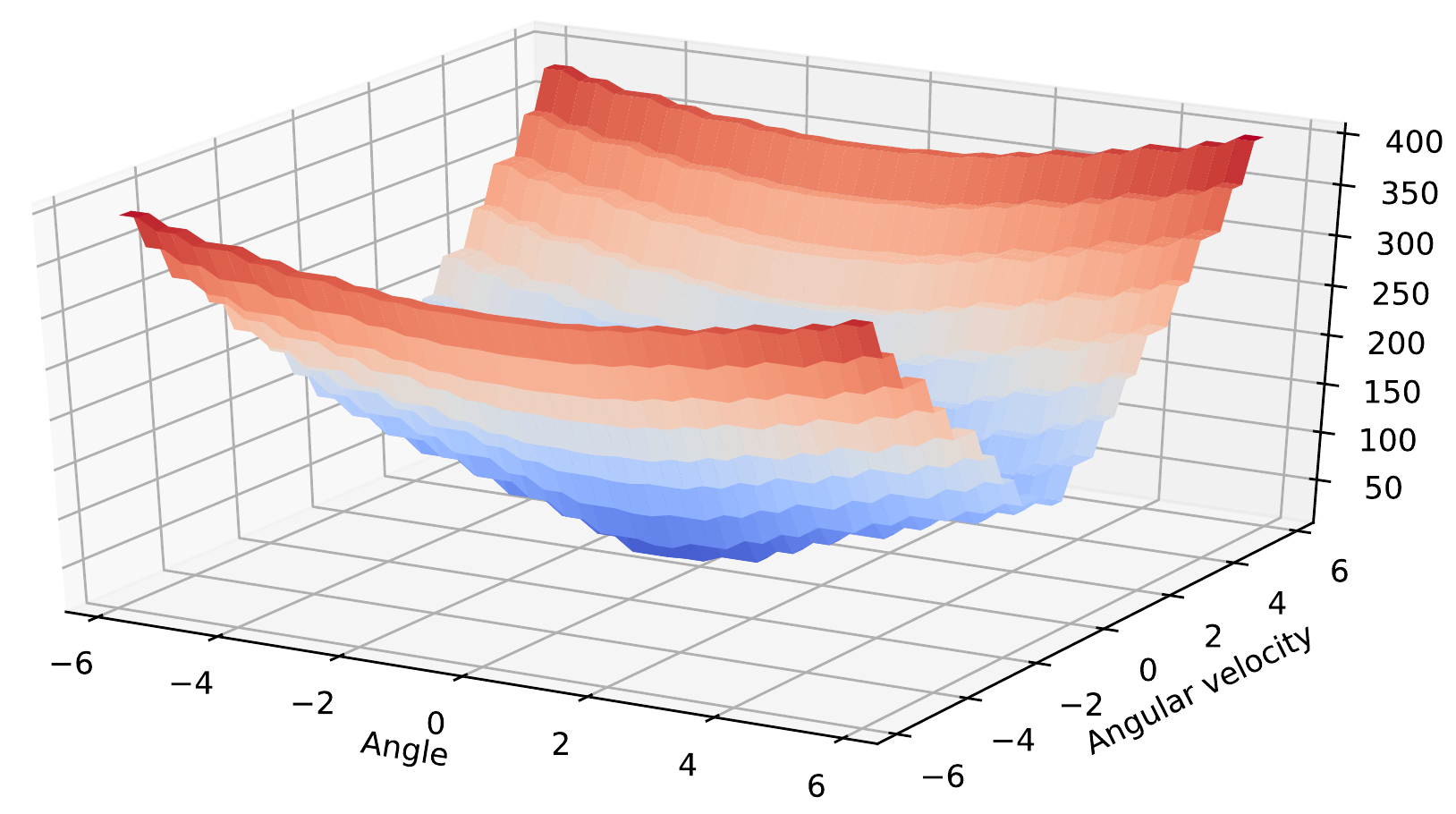}}%
		\quad
		\subfigure[]{\label{Fig:uncertainty1}%
			\includegraphics[width=0.2\linewidth]{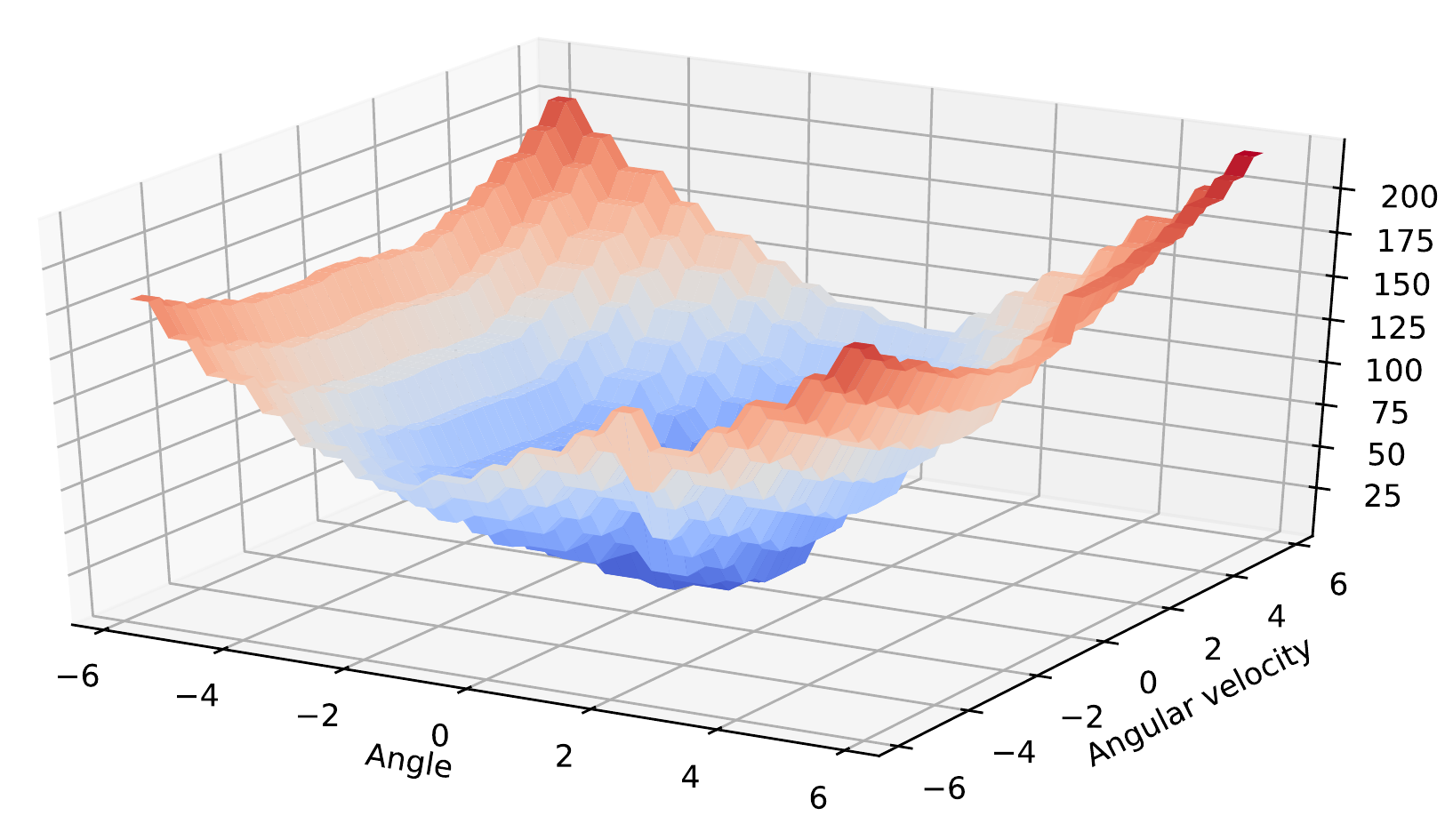}}
		\quad
		\subfigure[]{\label{Fig:uncertainty2}%
			\includegraphics[width=0.2\linewidth]{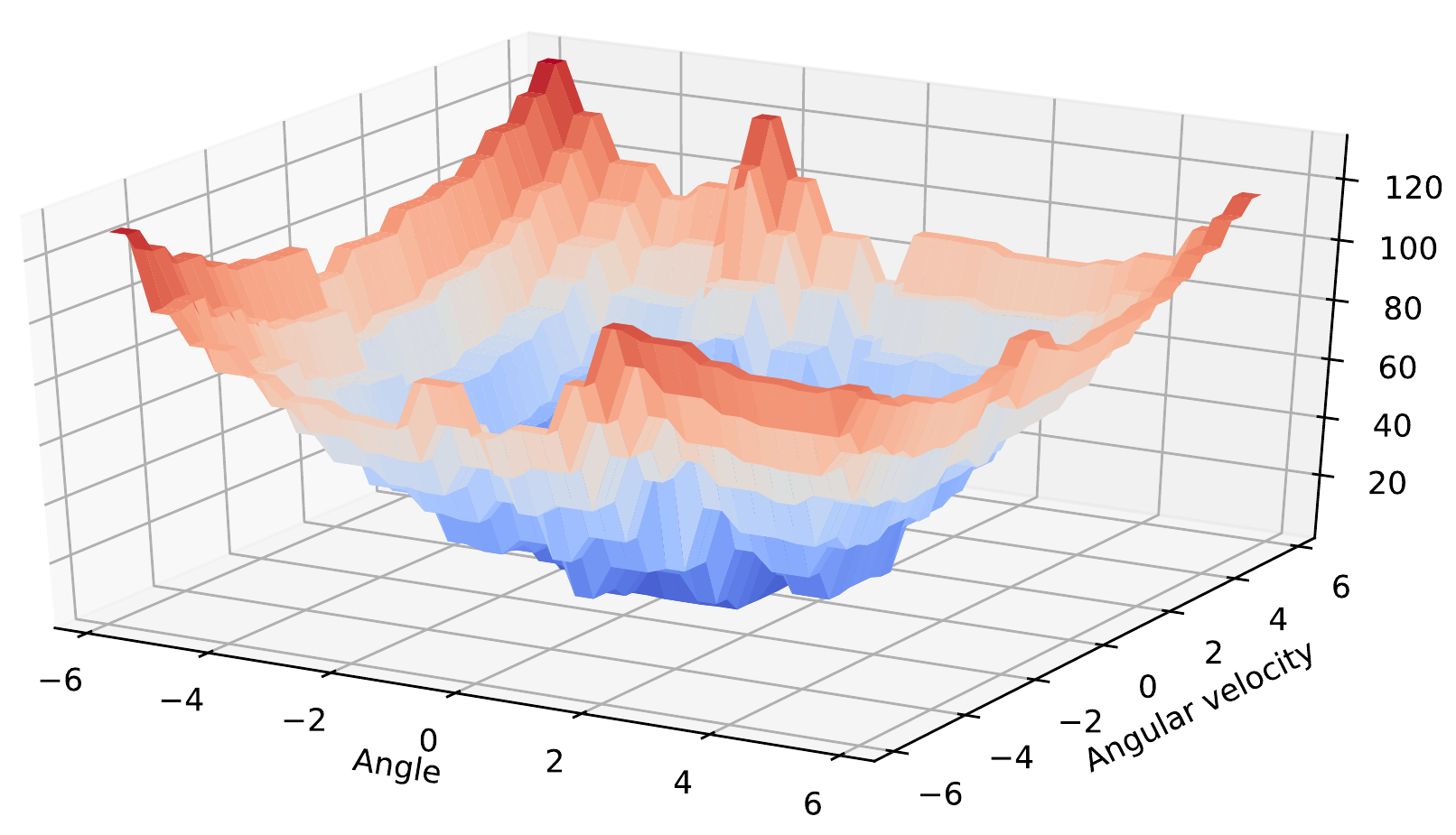}}
		\newline 
		\subfigure[]{\label{Fig:uncertainty3}%
			\includegraphics[width=0.2\linewidth]{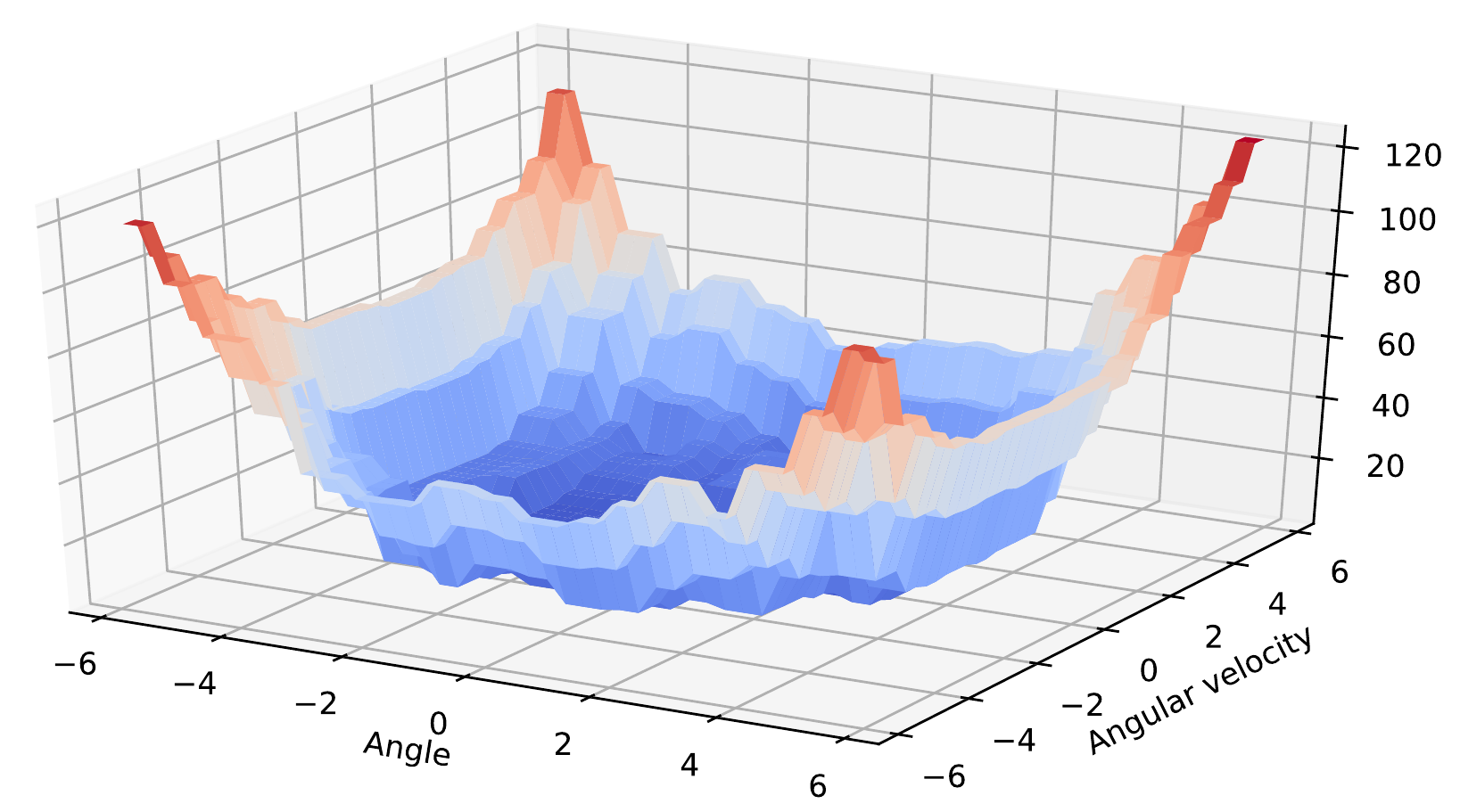}}%
		\quad
		\subfigure[]{\label{Fig:uncertainty4}%
			\includegraphics[width=0.2\linewidth]{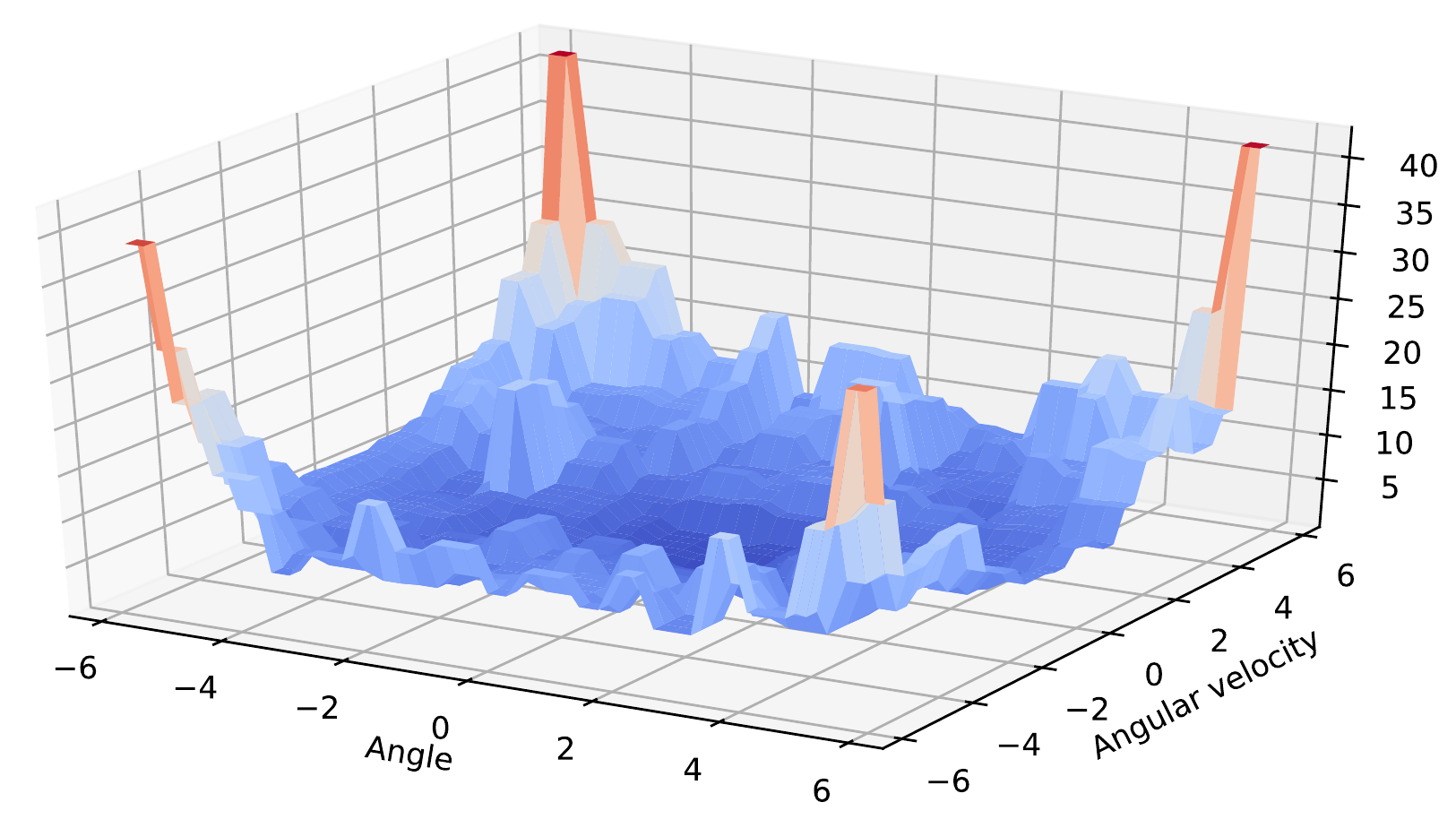}}
		\quad
		\subfigure[]{\label{Fig:uncertainty5}%
			\includegraphics[width=0.2\linewidth]{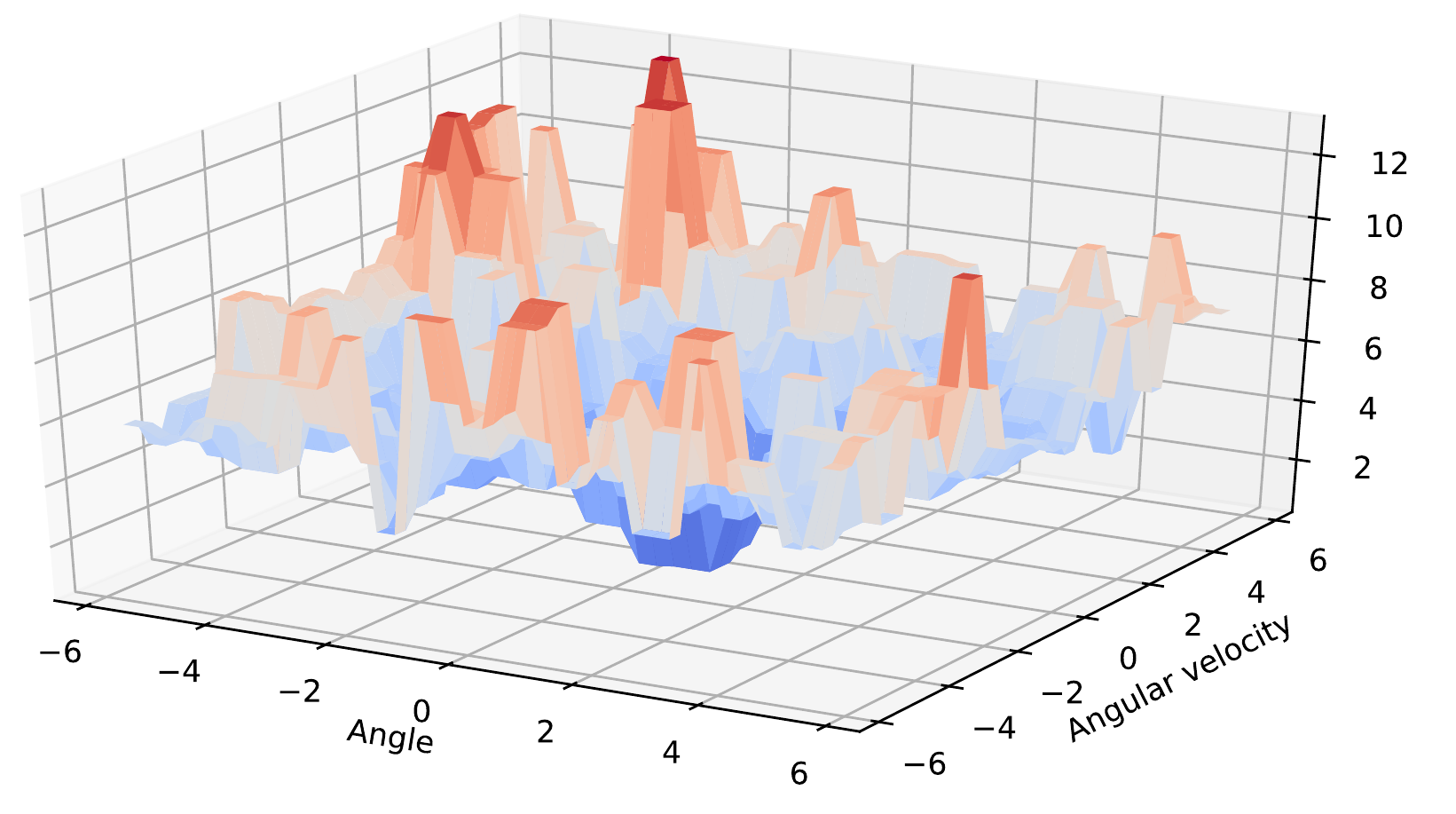}}
	}
\end{figure}

\begin{figure}[htbp]
	\floatconts 
	{fig:samples}
		{\caption{The step-by-step results of the sampling procedure and the sample gap obtained are provided that correspond to the results in Fig. \ref{Fig:learning_model} and \ref{Fig:uncertainty}. It is evident that by acquiring more samples over different steps and expanding the learning area, the sample gaps are decreased effectively. }}
	{%
	\hspace{12mm}	\subfigure[]{\label{Fig:samples0}%
			\includegraphics[width=0.2\linewidth]{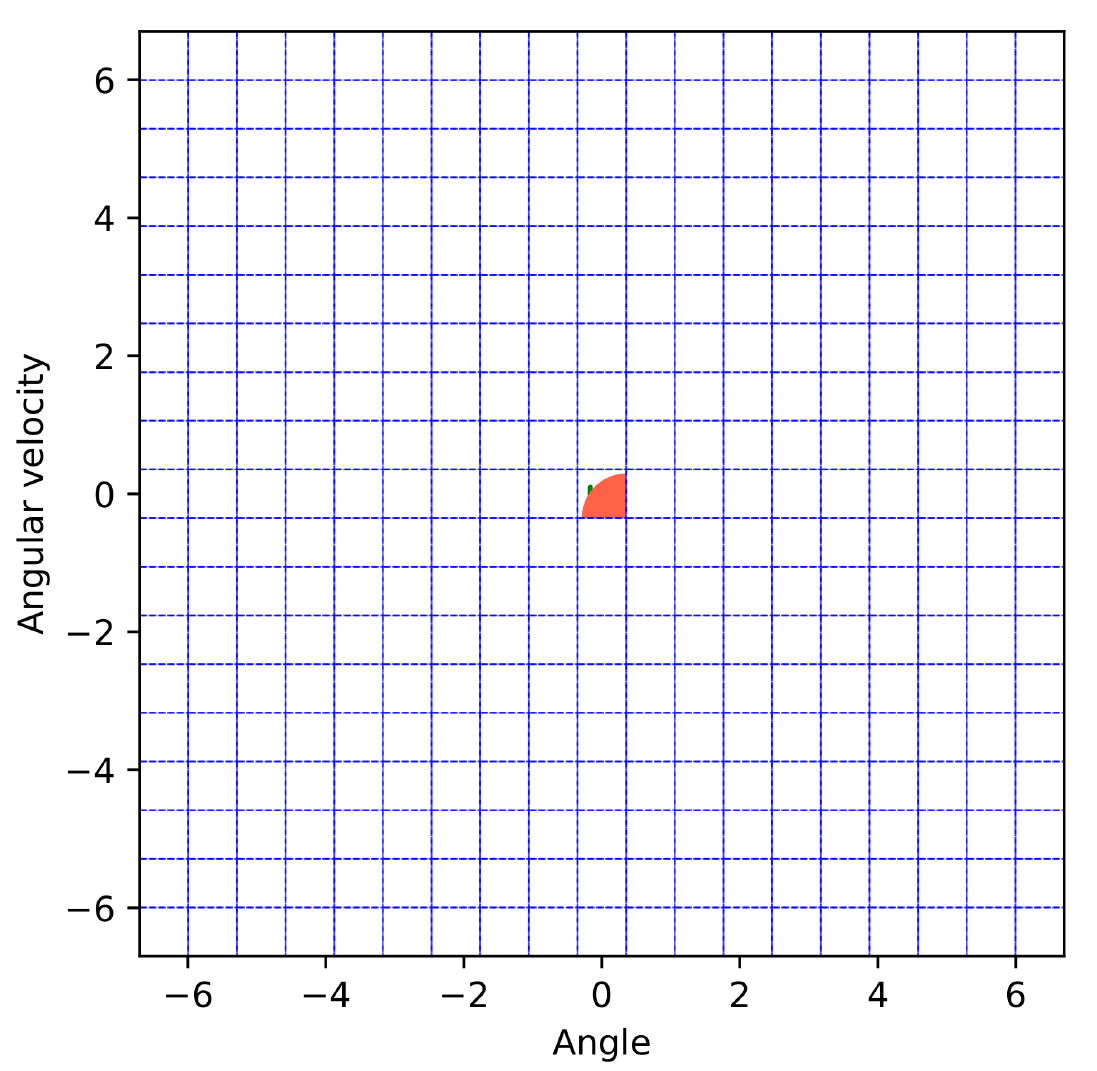}}%
		\quad
		\subfigure[]{\label{Fig:samples1}%
			\includegraphics[width=0.2\linewidth]{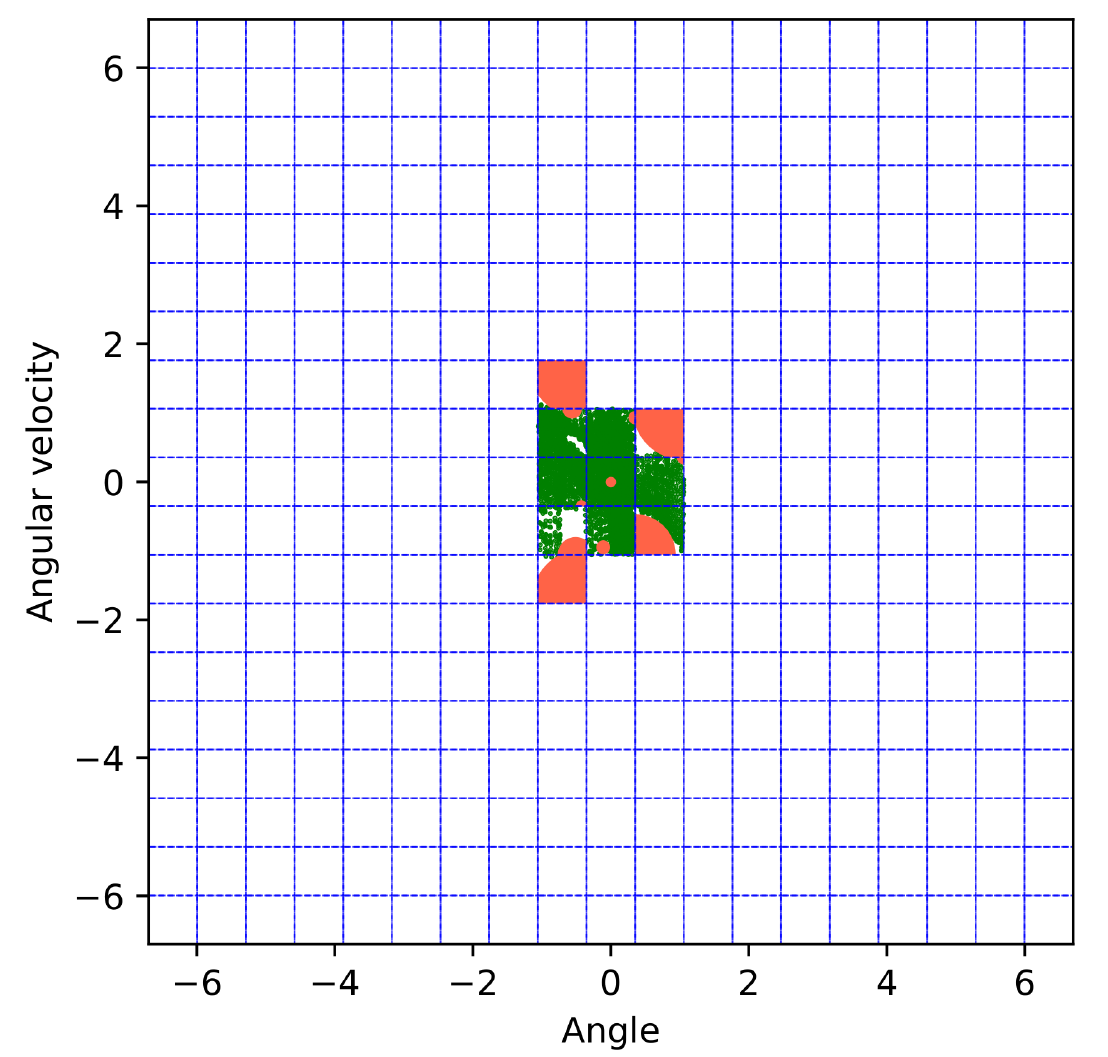}}
		\quad
		\subfigure[]{\label{Fig:samples2}%
			\includegraphics[width=0.2\linewidth]{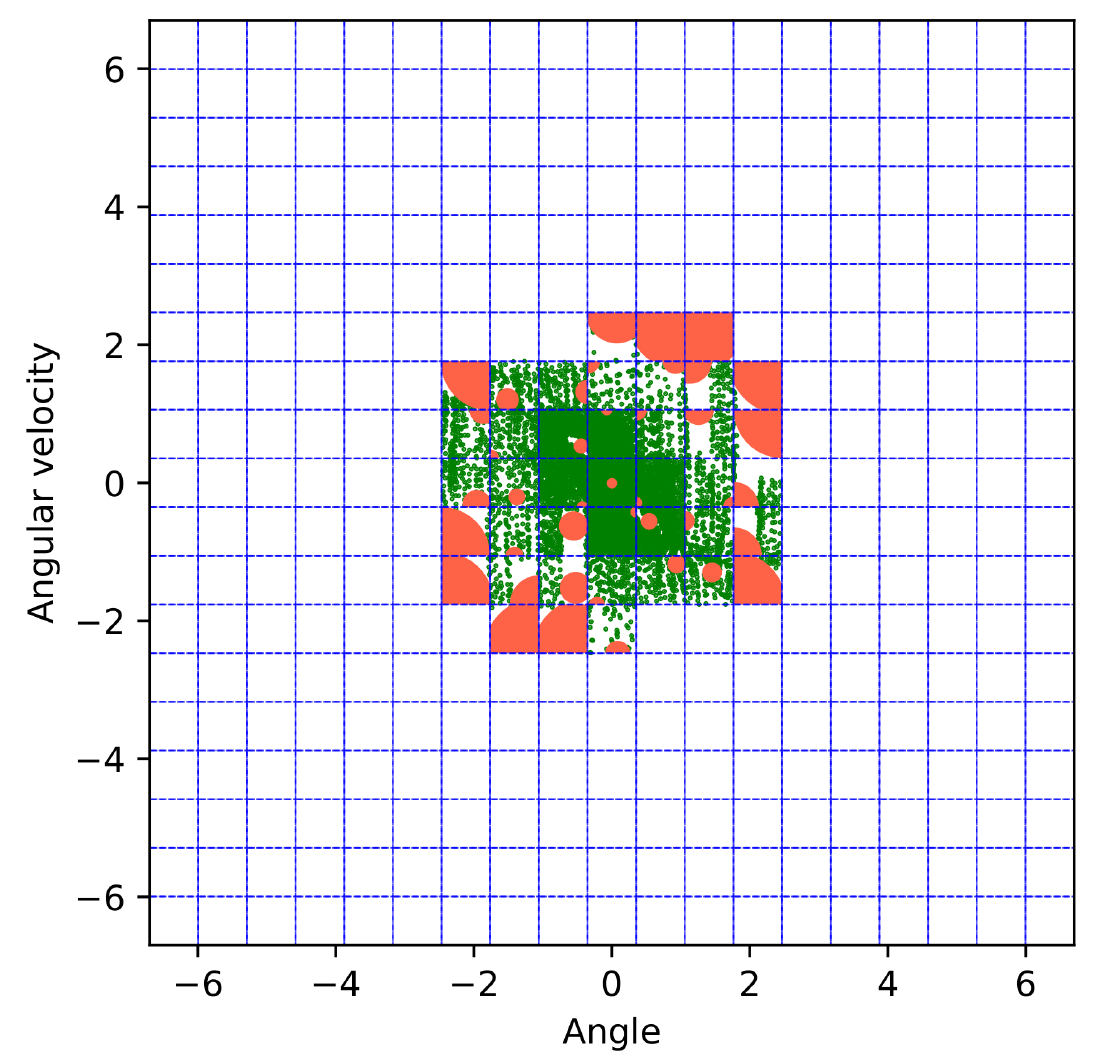}}
		\newline 
		\subfigure[]{\label{Fig:samples3}%
			\includegraphics[width=0.2\linewidth]{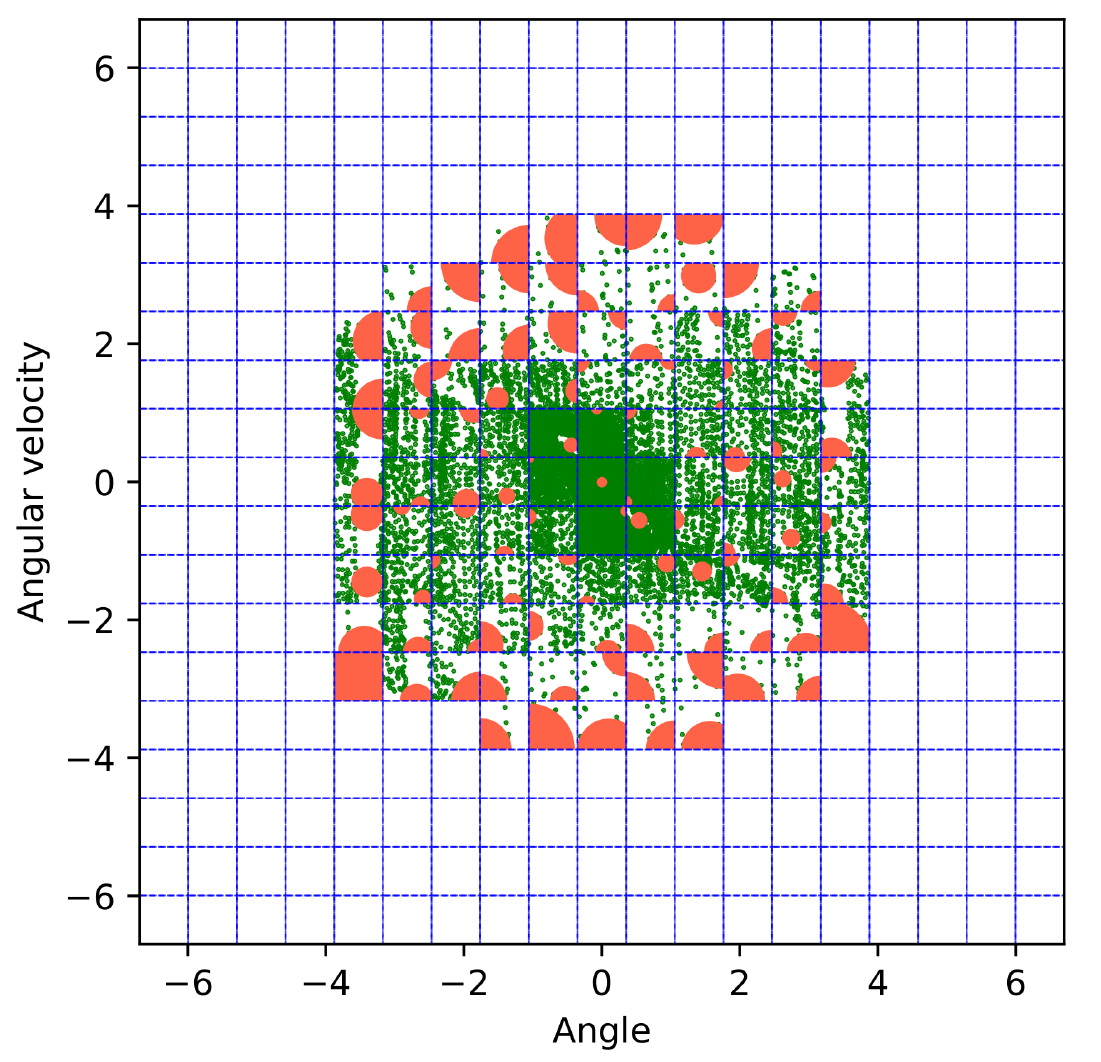}}%
		\quad
		\subfigure[]{\label{Fig:samples4}%
			\includegraphics[width=0.2\linewidth]{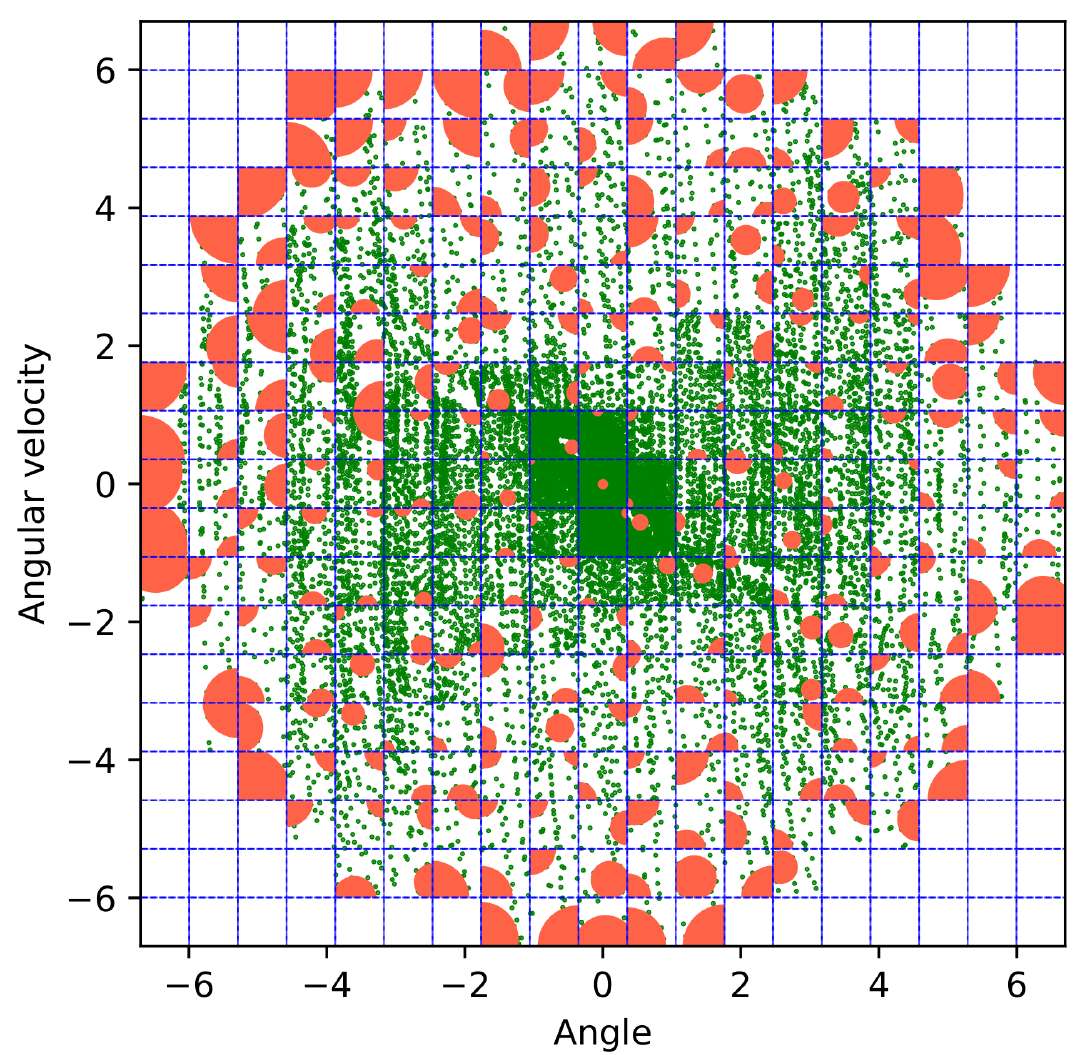}}
		\quad
		\subfigure[]{\label{Fig:samples5}%
			\includegraphics[width=0.2\linewidth]{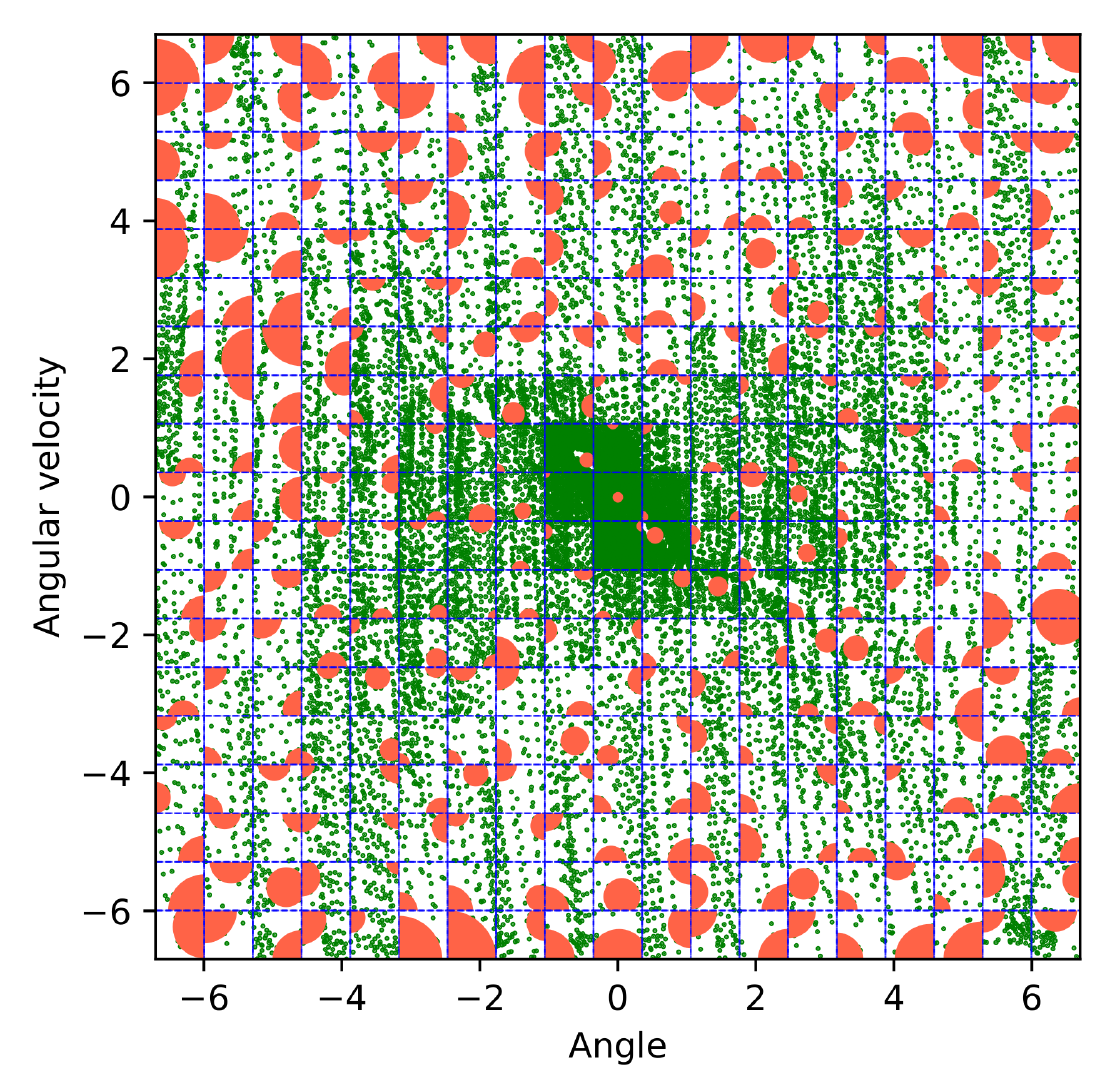}}
	}

\end{figure}

\subsection{Pendulum System} \label{apex:pend}
The state space description of the system is given as 
\begin{align}\label{SOL_pend}
&\dot{x}_1=x_2,\ \nonumber \\ 
&\dot{x}_2=\frac{G}{l}\sin(x_1)-\frac{b}{mL^2}x_2+\frac{1}{mL^2}u,
\end{align}

where the parameters are taken from \cite{chang2020neural} ($G$ = 9.81, 
    $L$= 0.5, $m$ = 0.15, $b$ = 0.1). The performance criteria are defined by the choices of $Q=\text{diag}([2,1])$, $R=1$.

\subsection{The Learned Matrix \texorpdfstring{$\hat{P}$}{Lg}}

\begin{align*}
 \hat P= \begin{bmatrix}
 0.69371067  &0.02892586 & 0.1944487   &0.05196313\\
  0.02892586 & 0.26941371  &0.02718769 &-0.21348358\\
  0.1944487   &0.02718769  &0.69518109  &0.05041737\\
  0.05196313 &-0.21348358  &0.05041737  &0.33469316\\
 \end{bmatrix}.
\end{align*}

\subsection{Dynamic Vehicle System}
\label{apex:vehicl}
 According to \cite{pepy2006path}, we present the dynamic model of the vehicle implemented. Let us define the states $x$ and $y$ as the coordinate of the center of gravity in the 2D space, $\theta$ as the orientation of the vehicle, $v_y$ as the lateral velocity, and $r$ as the rate of the orientation. Moreover, the input of the system is given by the front-wheel angle $\delta_f$. Then, by assuming a constant longitudinal velocity $v_x$, the dynamical model of the vehicle can be written as
\begin{align*}
&\dot v_y=-\frac{C_{\alpha f}\cos\delta_f+C_{\alpha r}}{mv_x}v_y+\frac{-L_fC_{\alpha f}\cos\delta_f+L_rC_{\alpha r}}{I_zv_x}r+\frac{C_{\alpha f}\cos \delta_f}{m} \delta_f,\\
&\dot r= (\frac{-L_fC_{\alpha f}\cos\delta_f+L_rC_{\alpha r}}{mv_x}-v_x)v_y-\frac{L^2_fC_{\alpha f}\cos\delta_f+L_r^2C_{\alpha r}}{I_zv_x}r+\frac{L_fC_{\alpha f}\cos \delta_f}{I_z} \delta_f,\\
&\dot x=v_x \cos \theta - v_y\sin \theta,\\
&\dot y=v_x \sin \theta + v_y \cos \theta,\\
&\dot \theta = r,
\end{align*}
where $C_{xf}$, and $C_{xr}$ denote the cornering stiffness coefficients of the front and rear wheels. Moreover, the distance of the center of gravity from the front and rear wheels are given by $L_f$, and $L_r$. 

\subsection{Numerical Results for the Dynamic Vehicle System}

In this section, to better demonstrate the merits of the algorithm proposed, we implemented the approach on a more complex system. The kinematic model of the vehicle system does not consider the real behavior of the system at high speeds where skidding is possible.  Therefore, \cite{pepy2006path} proposed a more realistic dynamical model of the vehicle which is implemented in this paper.
\subsubsection{Identify and Control}
Control objective is to minimize the distance of the vehicle from the goal point $(x,y)_{\text{goal}}=(70,70)$ in the 2D map. To achieve the objective, we run the vehicle from some random initial position and yaw values. Then, the identification and control procedures are done in a loop through different episodes. The longitudinal velocity of the vehicle is assumed to be constant in this system similar to \cite{pepy2006path}. Therefore, to minimize the cost given by the control objective, the vehicle converges to some circular path around the goal point, which is indeed the optimal path for the problem defined. Fig. \ref{Fig:results_vehicle}, contains the simulation results within an episode of learning, including the state and control signals, prediction error for each state, the value function and the modes.

\subsection{Comparison of Runtime Results}
To analyze the computational aspects of the proposed technique, we provide the runtime results while learning the dynamics and obtaining the control for both examples implemented. The proposed framework is considered as an online technique. Hence, in the applications, the computational complexity of the real-time identification and control becomes more important. Therefore, we here focus in the complexity of the online learning procedure rather than the verification technique which can be done offline.  Fig. \ref{fig:runtime} includes the runtime results separately for the identification and control units. Accordingly, the identifier and the controller can be updated in at most $1ms$ and $20ms$ respectively. Accordingly, it can be observed that for the higher dimensional system with an also larger number of partitions, the computations still remain in a tractable size that can allow real-time applications, considering the nature of the systems.
\begin{figure}[htbp]
	\floatconts
	{Fig:results_vehicle}
	{\caption{ ($a$). The state and control signals are illustrated within an episode of learning. It can be clearly seen from the position signals that the vehicle is able to minimize the distance from the goal point and converge to a circular path around the goal point $(x,y)_{\text{goal}}=(70,70)$ after some time of learning. \\ ($b$). The graph denotes the evolutions of the value function, the norm of the control parameters for the active mode, the prediction error, and the active mode of the piecewise model that correspond to the results in Fig \ref{PW_fig:vehicle1}. It can be seen that the value function learned is minimized.\\ ($c$). Corresponding to Fig. \ref{PW_fig:vehicle1} and \ref{PW_fig:vehicle2}, the prediction results of the learned model is compared with the original system within an episode of learning. It can be observed that the prediction signals shown by the black lines can match the ones obtained from the original dynamics. }}
	{%
		\begin{minipage}{1.05\textwidth} 
			\centering
		\hspace{-1cm}	\subfigure[]{\label{PW_fig:vehicle1}%
				\includegraphics[width=0.33\linewidth]{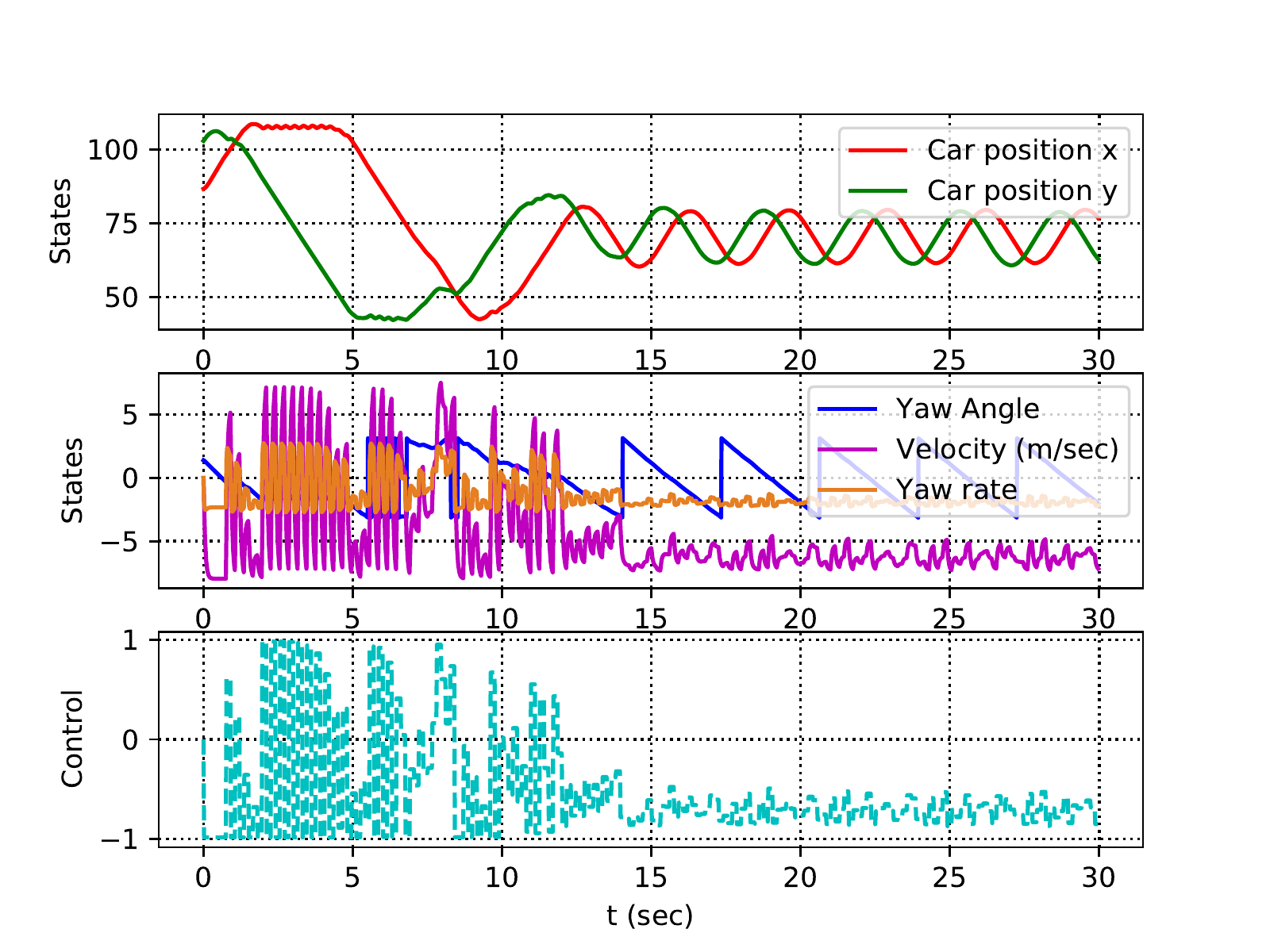}
			}%
			\subfigure[]{\label{PW_fig:vehicle2}%
				\includegraphics[width=0.3\linewidth]{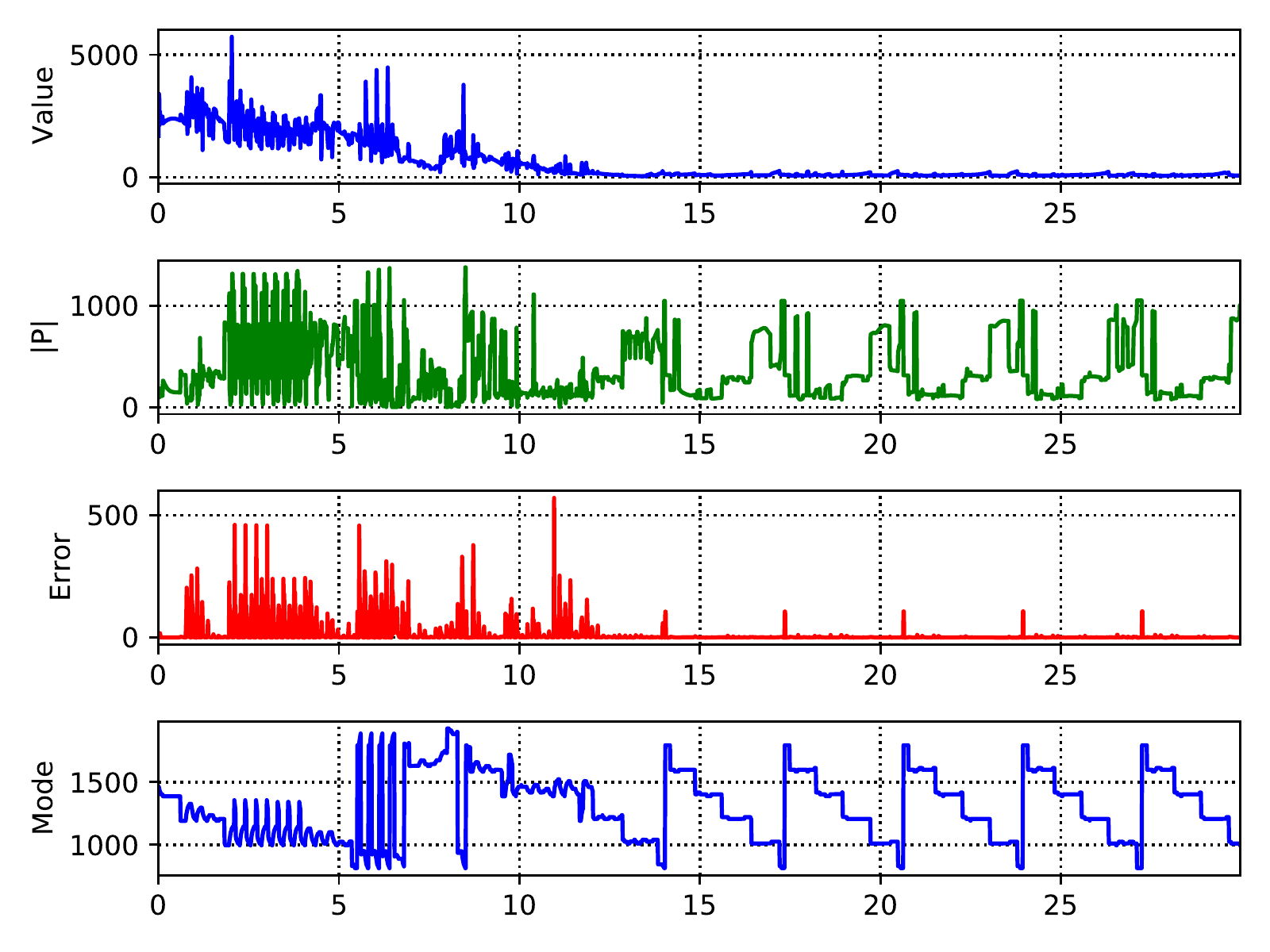}
			}			
			\subfigure[]{\label{PW_fig:vehicle3}%
				\includegraphics[width=0.3\linewidth]{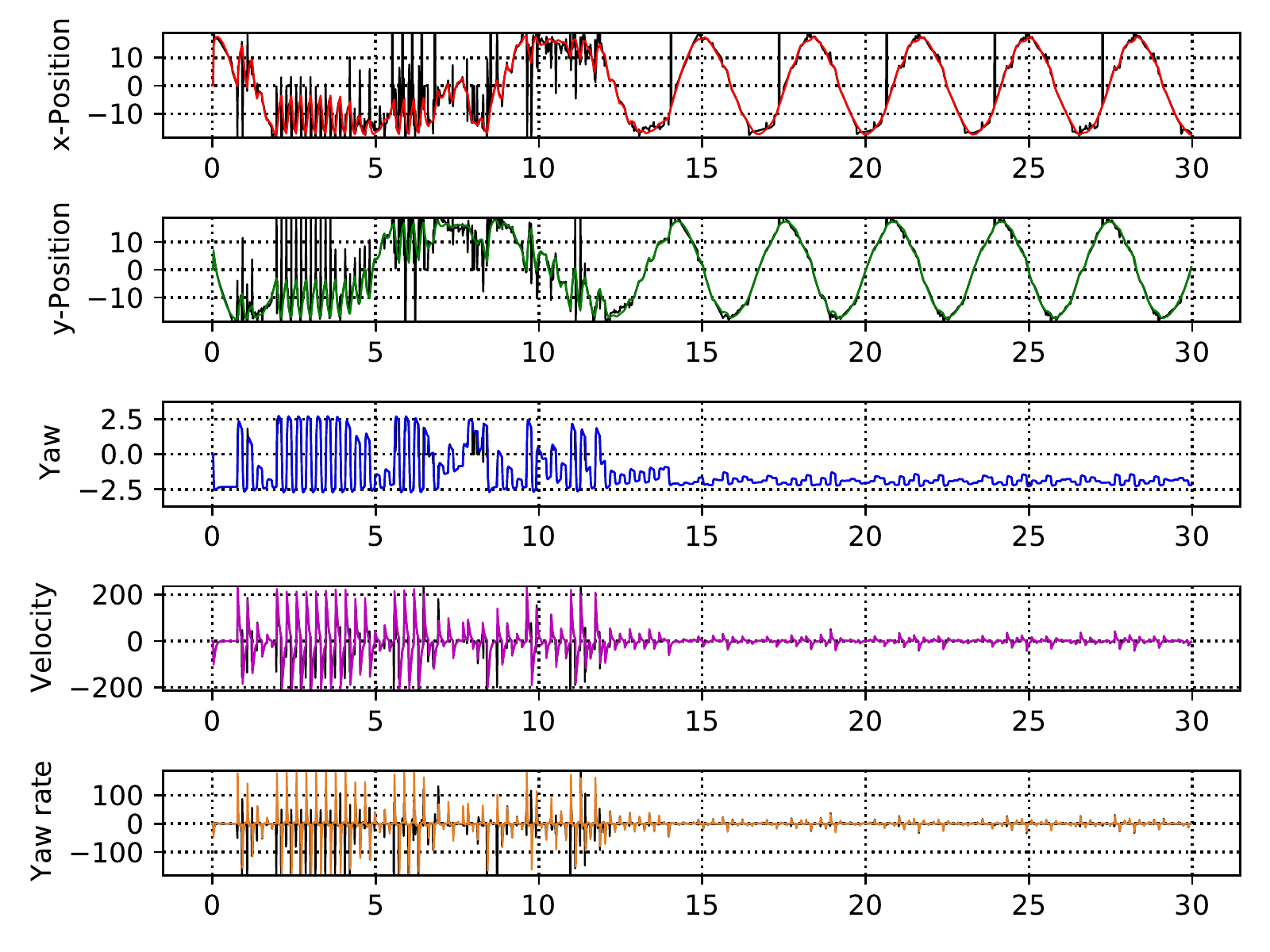}
			}

		\end{minipage}%
	}
\end{figure}

\begin{figure}[ht]
    \centering
    \includegraphics[width=10cm]{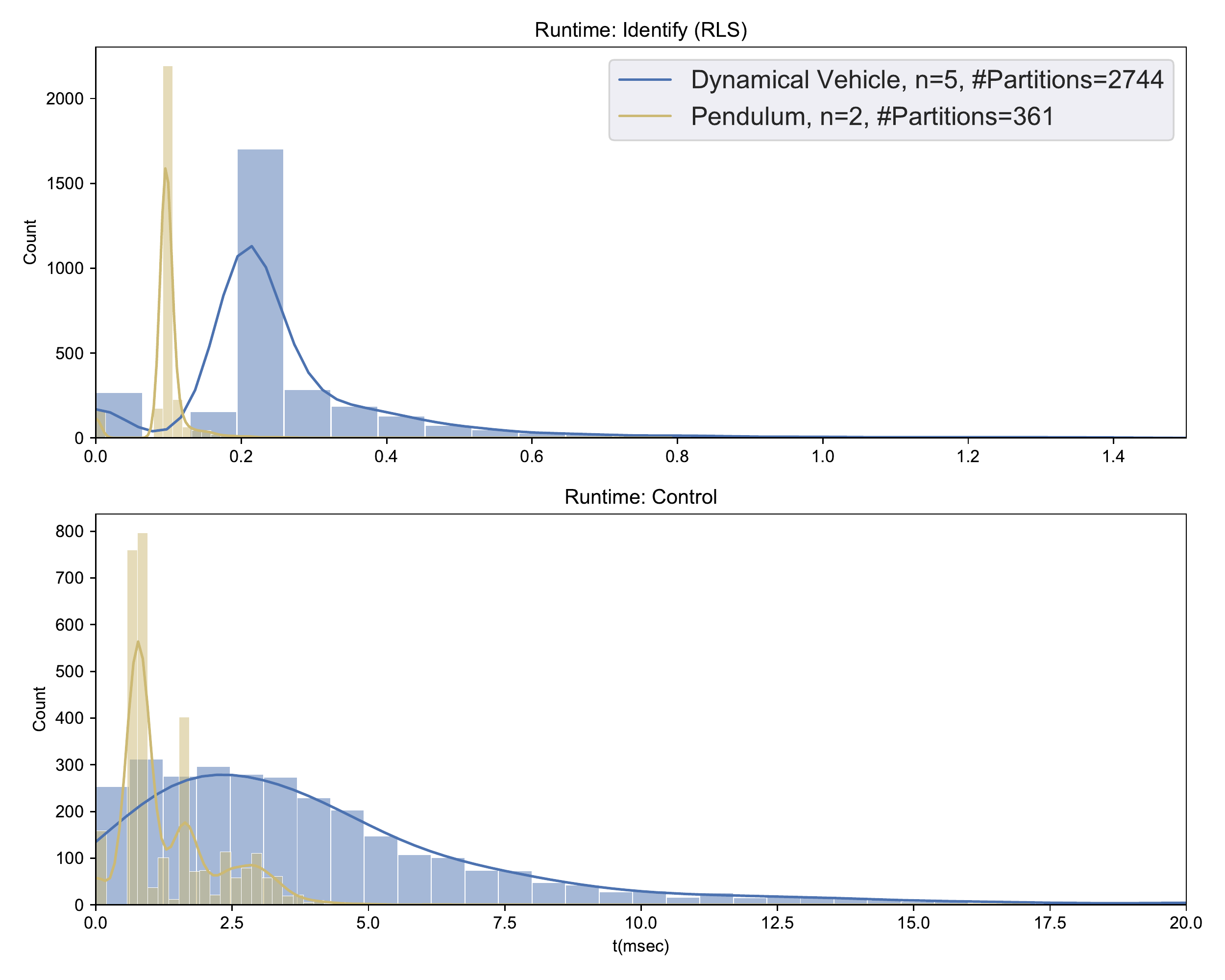}
    \caption{A comparison of the runtime results for the identification and control procedures separately is given for the implemented examples.  }
    \label{fig:runtime}
\end{figure}

\end{document}